\documentclass[a4paper,11pt]{amsart}

\usepackage{cite}
\usepackage{algpseudocode}
\usepackage{amsmath, amsthm}
\usepackage{textcomp}
\usepackage{amssymb}
\usepackage[all]{xy}
\usepackage{enumerate}
\usepackage{fancyhdr}
\usepackage{lscape}
\usepackage{url}
\usepackage{setspace}
\usepackage{longtable}
\usepackage{tikz}
\usepackage{wasysym}

\usepackage{longtable}

\usepackage{amsthm}
\theoremstyle{plain}

\newtheorem*{theorem*}{Main Theorem}

\theoremstyle{plain}
\newtheorem{theorem}{Theorem}[section]

\theoremstyle{definition}
\newtheorem{lm}[theorem]{Lemma}

\newtheorem{cor}[theorem]{Corollary}
\newtheorem{defi}[theorem]{Definition}

\theoremstyle{nonumberplain}

\theoremstyle{remark}
\newtheorem{rmk}{Remark}

\newcommand*\circled[1]{\tikz[baseline=(char.base)]{
           \node[shape=circle,draw,inner sep=1pt] (char) {#1};}}

\DeclareMathOperator{\Pic}{Pic}
\DeclareMathOperator{\Int}{Int}

\DeclareMathOperator{\Proj}{Proj}

\DeclareMathOperator{\rank}{rank}

\DeclareMathOperator{\Cox}{Cox}

\DeclareMathOperator{\N}{N}

\DeclareMathOperator{\Mob}{Mob}
\DeclareMathOperator{\MMob}{\overline{Mob}}
\DeclareMathOperator{\Amp}{Amp}

\DeclareMathOperator{\Nef}{Nef}
\DeclareMathOperator{\Bs}{Bs}

\DeclareMathOperator{\NE}{NE}
\DeclareMathOperator{\NNE}{\overline{NE}}
\DeclareMathOperator{\AAmp}{\overline{Amp}}

\DeclareMathOperator{\Aut}{Aut}
\DeclareMathOperator{\Bir}{Bir}

\DeclareMathOperator{\base}{Bs}
\DeclareMathOperator{\codim}{Codim}

\DeclareMathOperator{\supp}{Supp}
\DeclareMathOperator{\FLIP}{FLIP}

\newcommand{\C}{\mathbb{C}}
\newcommand{\Z}{\mathbb{Z}}
\newcommand{\HH}{\mathcal{H}}

\newcommand{\Q}{\mathbb{Q}}

\newcommand{\R}{\mathbb{R}}

\newcommand{\PP}{\mathbb{P}}
\renewcommand{\H}{\text{H}}
\renewcommand{\N}{\text{N}}
\newcommand{\x}{\mathrm{x}}
\newcommand{\s}{\mathrm{s}}
\newcommand{\y}{\mathrm{y}}
\newcommand{\z}{\mathrm{z}}
\renewcommand{\t}{\mathrm{t}}
\newcommand{\w}{\mathrm{w}}

\newcommand\qt{{\slash\kern-0.65ex\slash}}

\title{Mori dream spaces and Birational rigidity of Fano 3-folds}

\author{Hamid Ahmadinezhad}
\author{Francesco Zucconi}


\begin{document}

\begin{abstract} We highlight a relation between the existence of Sarkisov links and the finite generation of (certain) Cox rings. We introduce explicit methods to use this relation in order to prove birational rigidity statements. To illustrate, we complete the birational rigidity results of Okada for Fano complete intersection 3-folds in singular weighted projective spaces. 
\end{abstract}

\maketitle

\section{Introduction} Rationality question for varieties covered by rational curves (e.g.\,unirational varieties) has been a fundamental problem in algebraic geometry, the answer to which has opened various subjects across the field. The first examples of unirational varieties that are irrational were found about the same time by Clemens and Griffiths\,\cite{Grif}, Iskovskikh and Manin\,\cite{isk-manin}, Artin and Mumford\,\cite{Artin}. While the results of \cite{Grif} and \cite{Artin} indicate that the varieties under study (respectively, a smooth cubic 3-fold and some special singular quartic 3-folds in the weighted projective space $\PP(1,1,1,1,2)$) are irrational, the result in \cite{isk-manin} shows that a smooth quartic 3-fold is not birational to any conic bundle or fibration into del Pezzo surfaces, or any other Fano variety, including $\PP^3$, that is to say it is irrational in a strong sense. This latter phenomena is known as ``birational rigidity'', see Definition\,\ref{bir-rigid}. Birational rigidity also plays an important role in the birational classification of algebraic varieties.

Methods of Isovskikh and Manin have been improved in various directions and there are by now a handful of technical methods that can be used for proving birationally rigidity of a given variety. In this article, we take advantage of a factorisation property, known as Sarkisov program developed by Sarkisov, Corti, Hacon and McKernan \cite{Corti-Sar, HM-Sar}, and the notion of Mori dream spaces \cite{hu}, to give a new outlook on birational rigidity. The existing methods (of ``exclusion'') can be derived from this, this will be verified for some of them. We develop some explicit tools to use this new method in practice. As an illustration, we prove birational rigidity for some Fano 3-folds, and complete the question of birational rigidity for Fano 3-folds embedded, anticanonically, in singular weighted projective spaces that are quasi-smooth (their affine cone has only isolated singularities). In the next two subsections, we explain, first, the background on the problem together with our contribution, and then, second, we highlight the ideas behind the methodology.

\subsection{Birational classification}
The birational classification in dimension three saw a breakthrough in the eighties with effort of many people including Mori, Koll\'ar, Kawamata, Shokurov, Reid: Minimal Model Program (MMP for short) was proved\cite{mori}, see\,\cite{kollar-mori} for an introduction. Given a 3-fold $Z$, smooth or with mild singularities (terminal and $\Q$-factorial), MMP produces a birational model to $Z$ for which either the canonical divisor is nef or a Mori fibre space, when $Z$ is uniruled; see Definition\,\ref{Mfs} for the definition of a Mori fibre space. The simplest Mori fibre spaces are Fano varieties with Picard number one.

The classification (for the uniruled varieties) is reduced to studying Mori fibre spaces, and in particular Fano 3-folds, and relations among them. Birational rigidity for a Fano 3-fold means that its birational class consists of only one such Mori fibre space. A Fano 3-fold $X$ can be embedded in a weighted projective space via the anticanonical ring
\[R=\bigoplus_{n\in\Z}\H^0\left(X,-nK_X\right)\]
There are 95 families embedded as hypersurfaces, 85 families in codimension 2, 70 families in codimension 3 and so on, see\,\cite{database}. The result of Iskovskikh and Manin was generalised for general members of all 95 families by Corti, Pukhlikov and Reid\,\cite{CPR} and later was extended for any quasi-smooth member in these 95 families by Cheltsov and Park\,\cite{VJ-95}. The investigation of birational rigidity in codimension two families was initiated with the work of Iskovskikh and Pukhlikov\,\cite{pukh-c.i.}, where they prove that a general smooth complete intersection of a quadric and a cubic in $\PP^5$ is birationally rigid (see\,\cite[\S2.5 and \S2.6]{Pukh-book} for a detailed proof). This was later generalised by Okada \cite{okada} to show that within the remaining 84 families, a general member in 18 families are birationally rigid, while this property fails for all the remaining families. In this article we generalise this result to show that birational rigidity needs no generality assumption in Okada cases, and all is needed is the quasi-smoothness. 

\begin{theorem}\label{rigidity-Fano} Let $X$ be a quasi-smooth codimension two Fano 3-fold anticanonically embedded in a singular weighted projective space in one of the 85 families above. If a general member in its family is birationally rigid then $X$ is birationally rigid. 
\end{theorem}

Recall that it was proved in \cite{pukh-c.i.} that a general smooth $X_{2,3}\subset\PP^5$ is birationally rigid. Theorem\,\ref{rigidity-Fano} together with the result of \cite{VJ-95} and \cite{okada} show that the only remaining complete intersection Fano 3-folds, with $-K\sim\mathcal{O}(1)$, for which birational rigidity is yet to be discovered are $X_{2,2,2}\subset\PP^6$ and the special cases in $X_{2,3}\subset\PP^5$.

\subsection{Mori dream spaces and birational rigidity}
It is known that any birational map between Mori fibre spaces factors into a finite sequence of simpler maps (between Mori fibre spaces), known as Sarkisov links. It follows from the nature of these Sarkisov links, as we explore below in Section\,\ref{MDS}, that any Sarkisov link is given by a variation of geometric invariant theory (VGIT) for a group $(\C^*)^2$ acting on some affine space. It is well-known that having finitely many birational models through this VGIT is equivalent to the Cox ring, of a certain variety with Picard number two, being finitely generated. We show that existence of a Sarkisov link is equivalent to such finite generation and mobility of the anticanonical divisor of the variety with Picard number two, together with a technical (and necessary) assumption on the singularity of the birational models (see Theorem\,\ref{main-theoretical} for a precise statement). A general treatment of the behaviour of the anticanonical divisor against the mobile cone in relation with the birational models of the variety can be found in \cite[\S7]{Shokurov}. Proving that no Sarkisov link from a Mori fibre space to another one exists implies that it is birationally rigid.  

A Sarkisov link starting from a Fano variety with Picard number one can start only with a blow up. For a given Fano variety, and a quotient terminal singular point on it, we take advantage of Kawamata's theorem\,\cite{kawamata} and techniques developed in \cite{dP4} to give (most) generators of the Cox ring of the blow up variety of the singular point. For varieties that we study, we show the non-existence of Sarkisov links using this construction. Except in one case, we are able to show that the blow up variety is a Mori dream space; this does not follow from any known result and can be interesting in its own right.

The main advantage of the method we use is that it is systematic. Once it is run on a model it either excludes the possibility of a Sarkisov link or it provides a link to a new model. %On the other hand, the existing exclusion methods, while being very powerful, depend on the particular situation and a priori there is no way to decide which method to use (except with experience!). 
What we offer here can be used to exclude possibility of Sarkisov links systematically, with extra work, gaining extra rewards such as finding explicit birational maps from a Fano 3-fold to an elliptic fibration or a $K3$-fibration (e.g.\,\ref{full-Cox}), or it can be used as a tool to guide to which (known) method to use (e.g.\,\ref{f-mobility}).

\subsubsection*{Notation and convention}
We work over the field of complex numbers. Suppose $X$ is a $\Q$-factorial variety, the Neron-Severi group $\N^1(X)$ is the set of $\Q$-divisors on $X$ modulo numerical equivalence, and $\N_1(X)$ its dual cone of $1$-cycles modulo numerical equivalence. The following are various sub-cones in $\N^1(X)$ that we need.

\begin{center}
\begin{tabular}{ll}
The mobile cone & $\Mob(X)=\{D\in\N^1(X) |\codim\base(|D|)>1\}$\\
The nef cone & $\Nef(X)=\{D\in\N^1(X) | D.C\geq 0 \text{ for all curves }C\subset X\}$\\%\subset\Eff(X)$\\
The ample cone & $\Amp(X)=\{D\in\N^1(X) | D \text{ is ample }\}\subset\Nef(X)$ 
\end{tabular}
\end{center}

Also denote the cone of effective $1$-cycles by 
\[\NE(X)=\{\sum a_\alpha[C_\alpha] \text{ such that } a_\alpha\in\Q_{\geq 0}\}\subset\N_1(X)\] We have that $\Amp(X)=\Int(\Nef(X))$,
and $\AAmp(X)$ is the dual of $\NNE(X)$. See \cite{kollar-mori} for definitions and detailed study.

In this paper we use the notation FLIP for a flip, anti-flip or a flop, not necessarily in the Mori category. More precisely, a FLIP is a birational map $\FLIP\colon X\dashrightarrow X^\prime$, that is an isomorphism in codimension one, with a factorization 
\[\xymatrixcolsep{1.5pc}\xymatrixrowsep{2.3pc}
\xymatrix{
X\ar^{\FLIP}@{-->}[rr]\ar_{\psi}[rd]&&X^\prime\ar^{\varphi}[ld]\\
&Z&
}\]
such that $\varphi$, respectively $\psi$, is small contraction of extremal type corresponding to an extremal ray of $\NNE(X)$, respectively $\NNE(X^\prime)$. In particular, we have that $\rank\Pic(X/Z)=\rank\Pic(X^\prime/Z)=1$.

\subsubsection*{\bf Acknowledgments.} The authors would like to thank Cinzia Casagrande, Ivan Cheltsov, Okada Takuzo, Jihun Park, Vyacheslav Shokurov and  Konstantin Shramov for useful conversations and showing interest in this project. This project was initiated while the authors received support as researchers in pair from Trento. We would like to express our full gratitude to Marco Andreatta and CIRM (Trento) for the hospitality provided. A major part of this project was done while the first author visited Max Planck Institut f\"ur Mathematik in Bonn. We would like to specially thank MPIM for providing a great research environment, and particularly for the support provided for a short visit of the second author to MPIM.

\section{Mori dream spaces and Sarkisov links}\label{MDS}
In this section, we explore the connection between the notion of Mori dream spaces and the existence of Sarkisov links. Let us begin by defining the notion of pliability and related concepts.

\subsection{Pliability and birational rigidity}

\begin{defi}\label{Mfs} A Mori fibre space, denoted by Mfs, is a variety $X$ together with a morphism $\varphi\colon X\rightarrow Z$ such that
\begin{enumerate}[(i)]
\item $X$ is $\Q$-factorial and has at worst terminal singularities,
\item $-K_X$, the anticanonical class of $X$, is $\varphi$-ample,
\item and $X/Z$ has relative Picard number $1$, and
\item $\dim Z<\dim X$.
\end{enumerate}\end{defi}

For example, Fano varieties, that are $\Q$-factorial varieties with Picard number $1$ and ample anticanonical divisor with at worst terminal singularities, are a typical case of Mori fibre spaces.

\begin{defi}\label{bir-rigid} Let $X\rightarrow Z$ and $X^\prime\rightarrow Z^\prime$ be Mori fibre spaces. A birational map $f\colon X\dashrightarrow X^\prime$ is {\it square} if it fits into a commutative diagram\begin{center}$\xymatrixcolsep{3pc}\xymatrixrowsep{3pc}
\xymatrix{
X\ar@{-->}^f[r]\ar[d]& X^\prime\ar[d]\\
Z\ar@{-->}[r]^g&Z^\prime
}$\end{center} where $g$ is birational and, in addition, the map $f_L\!\colon\!X_L\!\dashrightarrow\!X^\prime_L$ induced on generic fibres  is biregular, where $L$ denotes the generic point of $Z$. In this case we say that $X/Z$ 
and $X^\prime/Z^\prime$ are {\it square birational}. We denote this by $X/Z\sim_{\text sq} X^\prime/Z^\prime$. A square birational map that is also biregular is called {\it square biregular}.
\end{defi}

\begin{defi}[\cite{corti-mella}]\label{pliability}The {\it pliability} of a Mori fibre space $X\rightarrow Z$ is the set
\[\mathcal{P}(X\slash Z)=\{\text{Mfs } Y\rightarrow T\mid X\text{ is birational to } Y\}\slash\sim_{\text sq}\] 

A Mori fibre space $X\rightarrow Z$ is said to be {\it birationally rigid} if $\mathcal{P}(X\slash Z)$ contains a single element. It is called {\it birationally super-rigid} if in addition $\Bir(X)=\Aut(X)$.\end{defi}

%\begin{defi}A Mori fibre space is birationally super-rigid if it is birationally rigid and $\Bir(X)$ is generated by its automorphisms and the fibrewise transforms.\end{defi}

\subsection{Sarkisov program} Sarkisov program is a method to factorise birational maps between two Mori fibree spaces. In other words, if $X\dashrightarrow X^\prime$ is a birational map between Mfs $X/Z$ and $X^\prime/Z^\prime$ then there is a decomposition
\begin{center}$\xymatrixcolsep{3pc}\xymatrixrowsep{3pc}
\xymatrix{
X\ar@{-->}^{\varphi_1}[r]\ar[d]&X_1\ar@{-->}^{\varphi_2}[r]\ar[d]&\cdots X_{n-1}\ar@{-->}^{\varphi_n}[r]\ar[d]&X_n\ar@{-->}^{\varphi_{n+1}}[r]\ar[d]& X^\prime\ar[d]\\
Z&Z_1&Z_{n-1}&Z_n&Z^\prime
}$\end{center}
where each $X_i/Z_i$ is a Mfs and each $\varphi_i$ is a Sarkisov link. See \cite[\S2]{Corti1} and \cite[Definition~3.4]{Corti-Sar} for details and explanation of Sarkisov links and the decomposition. The key point for our purpose is that each Sarkisov link is obtained by a geometric maneuver called a 2-ray game (\cite[\S2.2]{Corti1}). Here, we explain the rules of the game, with additional emphasis on the category it is plaid in.

\subsection{2-ray game}\label{2-ray-game}
We follow the construction in\,\cite{Corti1, Corti-Sar}. Suppose $X/S$ is a Mori fibre space. Let $Y$ and $T$ be two varieties with a morphism $Y\rightarrow T$ constructed in one of the following ways 
\begin{enumerate}[(1)]
\item either $T=S$ and $\varphi\colon Y\rightarrow X$ is a divisorial blow up in the Mori category,
\item or $Y=X$ and $T$ is the image of a map $\sigma\colon S\rightarrow T$, where $\sigma$ is one step of MMP on $S$.
\end{enumerate}
We have that $\rank\Pic(Y)-\rank\Pic(T)=2$, and hence $\NNE(Y/T)\subset\R^2$
\[\xygraph{
%{<0cm,0cm>;<1cm,0cm>:<0cm,1cm>::}
!{(0,0) }="a"
!{(0.3,1.7) }*+{R_1}="b"
!{(1.2,1)}*+{\NNE(Y/T)}="c"
!{(1.7,0.4) }*+{R_2}="d"
"a"-"b"  "a"-"d" 
}  \]
In particular, over the base $T$, there are at most two extremal projective morphisms, corresponding to $R_1$ and $R_2$. Assume $\varphi$, respectively $\psi$, correspond to $R_1$, respectively $R_2$. Because of our choice of $Y$, the ray $R_1$ is always contained in $\NE(Y/T)$. 

\begin{rmk} Note that $\NE(Y/T)$ may not be closed, i.e.\,$R_2\not\subset\NE(Y/T)$, in which case $\psi$ does not exist. %For example a general hypersurface of bi-degree $(4,3)$ in $\PP^1\times\PP^3$ is a Mori fibre space with Picard group of rank two and has this property, see \cite[Theorem~1.1,~(iv)]{ottem}. 
\end{rmk}

We play a game with the following rules:

{\bf Step\,1.} %If $\supp\{R_2\}$ contains no class of an effective 1-cycle, or 
If $\psi$ does not exist then we say the game terminates with an {\it error}. Suppose $\psi$ is a projective morphism. If $\dim\supp\{R_2\}>1$, then we stop and say the game {\it terminates}, where $\supp\{R_2\}$ is the support of all effective 1-cycles whose class belongs to the ray $R_2$. Otherwise, $\psi\colon Y\rightarrow Z$ is a (projective and) small contraction, contracting the support $\supp\{R_2\}$. In this case, if the FLIP does not exist, the game terminates with an error. Otherwise, there exists a variety $Y_1$ together with a projective small morphism $\varphi_1\colon Y_1\rightarrow Z$, defined over the base $T$, such that $\rank\Pic(Y_1/T)=2$. We go to the next step.

{\bf Step\,2.} By construction $\NNE(Y_1/T)\subset\R^2$, so there exist at most one more extremal projective morphism $\psi_1\colon Y_1\rightarrow Z_1$, over the base $T$, corresponding to the other extremal ray on $\NNE(Y_1/T)$.  Now we run Step\,1 for $Y_1$ with this new ray, and continue the process.

%Note that one of the two  boundary rays correspond to $\varphi_1$. 
Note that there is no guarantee that at each step the next move exists. %If all FLIPS exist we call this process an {\it extremal 2-ray game} played on $Y/T$.

\begin{defi} A 2-ray game that terminates after finitely many steps with no errors is called a {\it 2-ray link}. Hence, a 2-ray link terminates with contracting a divisor or an extremal fibration.
%We say that a 2-ray game is a {\it 2-ray link} if after finitely many steps the game terminates with
%\[\xymatrixcolsep{1.5pc}\xymatrixrowsep{2.3pc}
%\xymatrix{
%&Y_m\ar_{\varphi_m}[ld]\ar^{\psi_m}[rd]&\\
%Z_{m-1}&&Z_m
%}\]
%for some $m\geq 1$, and $\NE(Y_m/T)=\NNE(Y_m/T)$.
\end{defi}

\begin{lm}\label{2-ray-MDS} Let $Y$ and $T$ be as above. Then the 2-ray game on $Y/T$ is a 2-ray link if and only if the relative Cox ring of $Y$ over $T$ is finitely generated.
\end{lm}
\begin{proof} It follows from \cite[Proposition\,2.9]{hu} that the finite generation of the Cox ring is equivalent to $Y/T$ being a Mori dream space, see \cite[Definition\,1.10]{hu}. The fact that the first step of the 2-ray game exists (without any errors) is equivalent to $\Nef(Y)$ being the affine hull of finitely many semi-ample line bundles (and similarly at any step for $Y_i$). Finiteness of the game now implies there are only finitely many varieties isomorphic to $Y$ in codimension 1 such that the union of their nef cones gives the mobile cone of $Y$. Note that the latter is strictly forced by the fact that Picard number is 2.
\end{proof}

\begin{defi}\label{link} A finite extremal 2-ray game for which all the steps take place in the Mori category is called a {\it Sarkisov link}.\end{defi}

It is standard to make a case distinction when a 2-ray link is not a Sarkisov 2-ray link but it almost is, i.e.\,when the curves contracted by the last map of the 2-ray link, $\psi_m\colon Y_m\rightarrow Z_m$, are not $-K_{Y_m}$-positive but they are $-K_{Y_m}$-trivial, and all other steps of the game take place in the Mori category. In this case we call it a {\it bad link}. See \cite[\S5.5]{CPR}.

\paragraph*{\bf Sarkisov decomposition and extremal 2-ray link}
Given two Mfs $X/Z$ and $X^\prime/Z^\prime$ and a birational map $\varphi\colon X\dashrightarrow X^\prime$, we briefly explain how this can be decomposed into Sarkisov links via extremal 2-ray links, and we refer to \cite{Corti-Sar, Corti1} for proofs. 

One starts with a very ample complete linear system \[\HH^\prime=|-\mu^\prime K_{X^\prime}+A^\prime|\] on $X^\prime$, where $A^\prime$ is the pull back of an ample divisor on $Z^\prime$. Then take the $\HH$ to be the birational transform of $\HH^\prime$ on $X$ via $\varphi$. Since $X/Z$ is a Mfs then there exists a positive rational number $\mu$ and a divisor $A$ on $Z$ such that
\[\HH\subset |-\mu K_{X}+A|\]
If $\varphi$ is not square biregular then either $K_X+\frac{1}{\mu}\HH$ is not canonical or it is canonical but not nef (\cite[Corollary~2.5]{Corti1}). 
If it is not canonical then it has a base locus - the so-called maximal centre - with high multiplicity with respect to $\mu$ (\cite[Definition~2.6]{Corti1}). Now the 2-ray game is played by blowing up this centre and running the 2-ray game. On the other hand, if $K_X+\frac{1}{\mu}\HH$ is canonical but not nef then the 2-ray game is played by considering a suitable contraction $S\rightarrow T$ (see \cite[Case~2, Page 274]{Corti1} and do as explained above. Corti proves that both these games are actually 2-ray links in our notation \cite[Theorem~2.7]{Corti1}.

\section{Birational rigidity via mobile cone}
In this section, we explain in details our method of proving birational rigidity (via exclusion) based on the results in Section\,\ref{MDS}.

\subsection{Exclusions}\label{ex-methods}
The attempt to describe the pliability of a Fano variety $X$ is essentially reduced to studying the maximal centres on $X$. This is because the contraction map $Z\rightarrow T$ does not exist as $Z\!=\!\{\text{point}\}$.
Traditionally to prove birational rigidity for a Fano $3$-fold $X$ one takes $\Gamma\subset X$ as a possible centre, $\Gamma$ being an irreducible reduced subscheme of codimension $>1$, and proves that it is not a maximal centre. For example, one of the exclusion methods widely used is the following key lemma.
 
\begin{lm}[Lemma 5.2.1 \cite{CPR}]\label{main-CPR} If $\Gamma\subset X$ is a maximal centre and $\varphi\colon(E\subset Y)\rightarrow(\Gamma\subset X)$ is a maximal extraction then $(-K_Y)^2\in\Int\NNE(Y)$.
\end{lm}
The idea then is to find a nef divisor $M$ on $Y$ such that $M.(-K_Y)^2\leq 0$, hence 
\[(-K_Y)^2\notin\Int\NNE(Y)\]
This is often called the {\it test class method}. 

When the centre cannot be excluded, then proof follows by showing that any birational map $X\dashrightarrow X^\prime$ between the Fano $3$-fold $X$ and a Mori fibre space $X^\prime/Z^\prime$ has a Sarkisov decomposition where $X_1$ is isomorphic to $X$ but the map $\varphi_1\colon X\dashrightarrow X_1$ is a non-trivial birational self-map. Hence, any Sarkisov link from $X$ is a birational self-map of $X$. This is called the {\it untwisting method}. See for example \cite[\S4]{CPR}.\\

\paragraph*{\bf Mobile cone vs Mori cone for exclusion.} As explained above, one method of the proof of rigidity for Fano varieties, and more generally Mfs, follows by showing that $K^2$ does not belong to the interior of the Mori cone. Another way is the method of untwisting. However, computing the Mori cone is often rather a very difficult task, so is the untwisting, and a priori there is no systematic way of deciding which method to use.  Moreover, these methods do not produce much information about the geometry of the varieties. Here, we provide another method, dual to these, based on the study of the cone of effective divisors rather than effective 1-cycles, which is very systematic in nature. In particular we show the following theorem (c.f.\,\cite[\,Lemma\,2.8]{cAn}) for Fano varieties; for a more general statement see Theorem~\ref{main-theoretical}.

\begin{theorem}\label{idea-Fano} Let $X$ be a Fano variety. If $\Gamma\subset X$ is a maximal centre and $\varphi\colon(E\subset Y)\rightarrow(\Gamma\subset X)$ is a maximal extraction then $Y$ is a Mori dream space and $-K_Y\in\Int\MMob(Y)$.
\end{theorem}
\begin{proof} Suppose $\Gamma$ is a maximal centre, that implies $Y$ has a 2-ray link by the existence of the Sarkisov link \cite{Corti-Sar}. It follows from Lemma\,\ref{2-ray-MDS} that $Y$ is a Mori dream space. Now suppose that  $-K_Y\notin\Int\MMob(Y)$. This means that for any curve $C$ contracted at the last step of the 2-ray game, $-K_Y.C\leq 0$, which implies that the last map is not an extremal Mori contraction, hence the game goes out of the category.
\end{proof}

\begin{rmk}Lemma~\ref{main-CPR} follows from Theorem~\ref{idea-Fano}. This is because if $-K_Y\in\Int\MMob(Y)$, then $(-K_Y)^2\in\Int\NNE(Y)$. See for example \cite[\S1.5]{BCZ}. Hence, the method based on Lemma\,\ref{main-CPR} is a weaker method than the one by Theorem\,\ref{idea-Fano} for excluding maximal centres.
\end{rmk}

%The $K^2$ condition is used by finding a nef divisor $D$ for which $D.K^2\leq 0$, hence the implication that $K^2\notin\Int\NNE$. In fact, if the Cox ring of $Y$ is finitely generated then there are finitely many models, $Y_1,\cdots,Y_k$, isomorphic to $Y$ in codimension one and the same proof as the one for the remark above shows that if $K^2_i\notin\Int\NNE(Y_i)$ for any $i=1,\dots,k$ then $-K_Y\notin\Int\MMob(Y)$. This proves the following.

%\begin{theorem}[\bf Generalised test class method] Let $X$ be a $\Q$-factorial terminal Fano variety with Picard number one. Suppose $\varphi:Y\rightarrow X$ is an extremal extraction of a centre $\Gamma$ in $X$. If there exists a mobile divisor $D$ in $Y$ for which $K^2.D\leq 0$, then $\varphi$ cannot initiate a Sarkisov link. In other words, $\Gamma$ together with this extraction is excluded.
%\end{theorem}

\subsection{Toric embedding and 2-ray game}\label{2-ray}

Suppose $T$ is a toric variety with $\rank\Pic(T)=2$. Then it admits a dream 2-ray game (over $S=\{\text{point}\}$). Here we briefly explain how this game is played, and refer to \cite{BZ} for details.

The Cox ring of a rank 2 toric variety is of the form
\[\left(\begin{array}{ccccccc}
\x_0&\x_1&\x_2&\cdots&\x_{n-2}&\x_{n-1}&\x_n\\
a_0&a_1&a_2&\cdots&a_{n-2}&a_{n-1}&a_n\\
b_0&b_1&b_2&\cdots&b_{n-2}&b_{n-1}&b_n
\end{array}\right)\]
where the first row indicates the variables and the two other rows correspond to the weights of these variables. Note that the integers $a_i$ and $b_j$ are so that 1-dimensional cones (of the toric fan) corresponding to the variables $x_i$ span $\Z^2$, each cone is convex and extremal, and moreover \[d_A=\gcd(\text{ all nonzero determinant of } 2\times 2\text{ minors})\neq 1\]
furthermore assume that the variables, from left to right, are in order so that the GIT chamber of this toric variety is of the following form:
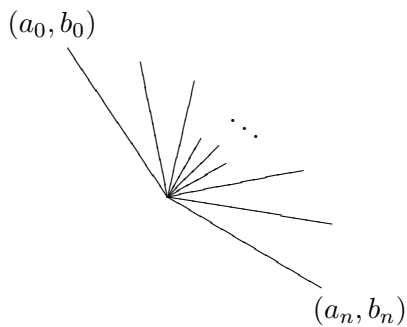
\[\xygraph{
%{<0cm,0cm>;<1cm,0cm>:<0cm,1cm>::}
!{(0,0) }="a"
!{(-1.2,1.8) }*+{(a_0,b_0)}="f"
!{(-0.3,1.5) }*+{}="e"
!{(0.3,1.3) }*+{}="y"
!{(0.4,0.7) }*+{}="k"
!{(0.8,0.8) }*+{\ddots}="z"
!{(0.7,0.4) }*+{}="l"
!{(1.5,0.3) }*+{}="x"
!{(1.8,-0.3) }*+{}="b"
!{(2,-1.2) }*+{(a_n,b_n)}="d"
"a"-"b"  "a"-"e" "a"-"f" "a"-"d" "a"-"x" "a"-"y" "a"-"z" "a"-"k" "a"-"l"
}  \]

Denote by $D_i$ the divisor $(\x_i=0)$ in $T$. We have that
\begin{enumerate}[(i)]
\item $\Mob(T)=\left<D_1,D_{n-1}\right>$, if $\frac{a_0}{b_0}\neq\frac{a_1}{b_1}$ and $\frac{a_{n-1}}{b_{n-1}}\neq\frac{a_n}{b_n}$,
\item $\Mob(T)=\left<D_0,D_{n-1}\right>$, if $\frac{a_0}{b_0}=\frac{a_1}{b_1}$ and $\frac{a_{n-1}}{b_{n-1}}\neq\frac{a_n}{b_n}$,
\item $\Mob(T)=\left<D_0,D_n\right>$, if $\frac{a_0}{b_0}=\frac{a_1}{b_1}$ and $\frac{a_{n-1}}{b_{n-1}}=\frac{a_n}{b_n}$.
\end{enumerate}
Moreover, $\MMob(T)$ decomposes into a finite number of subcones
\[\MMob(T)=\bigcup_{i=1}^{k}\Nef(T_i)\]
where the $\Nef(T_i)$ are in order so that the ample cone of $T$ is $\Amp(T)=\Int(\Nef(T))$.
In this notation $T_i$ are the varieties isomorphic to $T$ in codimension 1. They are obtained by the variation of GIT, as follows: the game is played by choosing a character in each chamber at each step, going clockwise. The first map, either a fibration (when $a_0/b_0=a_1/b_1$) or a contraction (when $a_0/b_0\neq a_1/b_1$), is when the character is chosen to be $a_1/b_1$. Same happens at the far end of the game for $a_n/b_n$.  The variation of GIT then dictates the divisorial 2-ray game. For example passing above the ray of $\x_i$ we have a FLIP. If $k$ is the number of variables that lie on the ray as $\x_i$ then the base of this FLIP has dimension $k-1$. The number of variables before, and respectively after, $\x_i$ indicate the contracted loci before, and respectively after, the FLIP. For instance, if $k=2$ and there are 3 variables before this ray and 4 after then a 3-dimensional locus on the left hand side contracts (as fibration) to a line and a 4-dimensional locus contracts, to the same line, on the right hand side.

\begin{rmk}\label{toric-antiflip} Note that even if we start the 2-ray game from a toric Mori fibre space, that is by construction a 2-ray link, there is no guarantee that it is a Sarkisov link even if $-K_T\in\Int\Mob(T)$. For example, consider the toric variety $T=\Proj_{\PP^1}\mathcal{E}$, where $\mathcal{E}=\mathcal{O}_{\PP^1}\oplus\mathcal{O}_{\PP^1}(2)\oplus\mathcal{O}_{\PP^1}(2)$. The Cox ring of $T$ is
\[\left(\begin{array}{ccccccc}
0&0&1&1&1\\
1&1&0&-2&-2
\end{array}\right)\]
and the 2-ray game starts from a $\PP^2$ fibration over $\PP^1$, and exhibits a non-terminal anti-flip to a $\PP(1,1,2)$ fibration over $\PP^1$, which does not take place in the Mori category. Note that the mobile cone is generated, in $\Q^2$, by the classes $(0,1)$ and $(1,-2)$, which clearly includes $-K_T$ in the class $(3,-2)$ as an interior point. In fact we have the following theorem.
\end{rmk}

\begin{theorem}\label{main-theoretical} For a Mori fibre space $X/S$, the 2-ray game played on $Y/T$, where $Y$ and $T$ are constructed as in\,\ref{2-ray-game}, is a Sarkisov link if and only if the relative Cox ring of $Y$, over $T$, is finitely generated, $-K_{Y/T}\in\Int\MMob(Y/T)$ and all the anti-flips in the 2-ray game are terminal.
\end{theorem}
\begin{proof} Note that in the process of 2-ray game anti-flips, flops and flips can all occur, in that order. However, it follows from the rules of minimal model program that if the variety before flip (or flop) is in the Mori category then the variety after the flip is also in the category. So the condition on the anti-flips (being terminal) controls the case where anti-flip can go out of the Mori category, as for example in Remark\,\ref{toric-antiflip}. The rest of the proof follows precisely the proof of Theorem\,\ref{idea-Fano} in the relative set up.
\end{proof}

\begin{cor} Let $X/S$ be a 3-fold Mori fibre space with $\rank\Pic(X)=2$. Then it is birationally super-rigid if either it is not a Mori dream space or it is a Mori dream space and
\begin{enumerate}
\item $-K_X\notin\Int\overline{\Mob(X)}$, and
\item for any extremal divisorial extraction $\widetilde{X}$ of $X$, at a centre $C\subset X$, either $\widetilde{X}/S$ is not a Mori dream space or it is a Mori dream space with $-K_{\widetilde{X}/S}\notin\Int\overline{\Mob(\widetilde{X}/S)}$.
\end{enumerate}
\end{cor}

\begin{defi}\label{follow} Let $T$ and $Y$ be $\Q$-factorial projective varieties with Picard number $2$, and suppose $Y\subset T$. We say that {\it $Y$ 2-ray follows $T$} if the 2-ray game of $T$ restricts to a dream 2-ray game on $Y$.
\end{defi}

\begin{rmk} By Lemma\,\ref{2-ray-MDS}, if $Y$ is a Mori dream space with $\rank\Pic(Y)=2$, then it admits a 2-ray link. In this case the game is the restriction of the 2-ray game of a toric variety to $Y$, c.f. \cite[Proposition~2.11]{hu}.

It follows that if $Y$ is a Mori dream space with finitely generated Cox ring 
\[\Cox(Y)=\frac{\C[x_1,\dots,x_n]}{I}\]
 then $Y$ 2-ray follows $T$, where $T$ is the toric variety with $\Cox(T)=\C[x_1,\dots,x_n]$, with the same grading as $Y$, such that $Y$ is equivariantly embedded in $T$. This follows from \cite[Proposition~2.11]{hu} and its proof.
\end{rmk}

\subsection{Unprojection and generators of the Cox ring}
Let us describe how we will use the methods introduced above. We start by a Fano 3-fold  $X$ embedded (anticanonically) in a weighted projective space $\PP$. The aim is to exclude a point $p\in X$, with quotient terminal singularity. We construct a toric blow up $T\rightarrow \PP$ centred at $p$, and consider $Y$, the proper transform of $X$ under the blow up. The construction of $T$ is done in such way that $Y\rightarrow X$ is an extremal blow up of $X$ at $p$. For the purpose of the applications in this article we need only  the classification of extremal extractions (in dimension 3) due to Kawamata\,\cite{kawamata}. For other types of terminal singularities, i.e.\,non-quotient, one can use the results of Kawakita\,\cite{kawa1,kawa2,kawa3,kawa4}. First note that $-K_Y$ can be computed immediately using adjunction formula. Then we run the 2-ray game of $T$ and restrict it to $Y$. If $Y$ 2-ray follows $T$, then we have a full description of its Cox ring, and in particular its mobile cone. A simple check will tell us whether $-K_Y$ is interior or not. If $Y$ does not 2-ray follow $T$, then by means of unprojection (of $Y$ away from a ``fake'' divisor -- a divisor that lies entirely in the unstable locus of the GIT quotient) we embed $Y$ in a bigger toric variety $T'$ (in higher codimension), by adding generators to the Cox ring. Then we do the 2-ray game on $T'$ and so on. No matter if this process terminates with a suitable toric embedding or not we are successful to show that $-K_Y$ is not mobile in all cases, using a few technical tools that we introduce later on. Hence we obtain the desired exclusion.

Let us recall Kwamata's description of extremal blow ups.

\subsubsection{Kawamata blow up}\label{kawamata} See \cite{kawamata}. Suppose that $\varphi\colon Y\rightarrow X$ is a divisorial contraction to a centre $\Gamma\subset X$, and suppose $p\in X$ is a terminal quotient singularity with the germ $\frac{1}{r}(1,a,r-a)\text{ such that }r\geq 2\text{ and }a\text{ is coprime to }r$.
If $p\in\Gamma$ then $\Gamma=\{p\}$ and $\varphi$ is the weighted blow up with weights $(1,a,r-a)$ and we have that
\[K_Y=\varphi^*(K_X)+\frac{1}{r}E\]
where the irreducible scheme $E\simeq\PP(1,a,r-a)$ is the exceptional divisor of the blow up. We call this map the {\it Kawamata blow up} of $X$ at $p$.

\subsubsection{Construction of the toric blow up}\label{toric-method}
We explain here how the toric variety $T$, the blow up of the weighted projective space, is constructed. We do this for the case of a Fano hypersurface (of Iano-Fletcher--Reid~\cite{fle}). The extension to higher codimensions is very similar and the differences will be explained later, once they appear.

We roughly follow the notation of \cite{CPR} to ease reading. We denote a Fano hypersurface \[X_d\colon(f=0)\subset\PP(1,a_1,a_2,a_3,a_4)\] where $a_1, a_2, a_3, ,a_4$ are positive integers and $d$ is the degree of $f$, the polynomial that defines $X$, such that this description is well formed, see~\cite[Theorem~14.1]{fle}. The variables of $\PP(1,a_1,a_2,a_3,a_4)$ will be denoted by $x,y,z,t,s$ when all $a_i$ are known and $1\leq a_1\leq \dots\leq a_4$. When the values of $a_i$ are not specified, we use $x_0,\dots,x_4$ to denote the variables.

\begin{enumerate}[(1)]
\item Start by a concrete homogeneous polynomial $f$, of degree $d$, defining the Fano $3$-fold $X$ as a hypersurface in the weighted projective space $\PP=\PP(1,a_1,a_2,a_3,a_4)$. Quasi-smoothness implies that $f$ imposes no singularities on $X$.
\item Consider a singular point $p\in X$, a quotient terminal singularity, that is of type $1/a_4(a,a_4-a,1)$, where $a_4$ and $a$ are two coprime integers, with $a_4>a$. Note that we have assumed this corresponds to the point $p_4\!=\!(0\!:\!0\!:\!0\!:\!0\!:\!1)$, and that here we have no assumptions on the order of $a_i$. In order to fulfill the assumption of quasi-smoothness at this point, the polynomial $f$ must include a monomial of type $x_4^kL$, where $k$ is a positive integer and $L$ is a linear form in some other variables. Without loss of generality, assume that $x_4^kx_3\in f$. We will call $x_3$ the {\it tangent variable} at $p$ and $x_4$ the {\it nonvanishing variable} at $p$.

\item Consider the variety with the Cox ring $\C[u,\x_4,\x_3,\x_2,\x_1,\x_0]$ of rank two with irrelevant ideal $I=(u,\x_4)\cap(\x_0,\dots,\x_3)$ and the grading
\[\left(\begin{array}{cccccc}
u&\x_4&\x_3&\x_2&\x_1&\x_0\\
0&a_4&a_3&a_2&a_1&1\\
-a_4&0&\alpha&a&a_4-a&1\end{array}\right)\]
For different choices of $\alpha$, congruent to $a_3 \mod a_4$, we get a Deligne-Mumford toric stack, whose well-formed toric model, $T$, is easy to obtain by subtracting the first row from the second and then dividing the (new) second row by $-a_4$; see \cite{dP4} for details and general recipe. Note that the weight of the variable $\x_4$, for instance, becomes $(a_4,1)$. 

\item For the choice of the character $(a_4,1)$ (in the new matrix) in the variation of the Geometric Invariant Theory (GIT), there is a birational morphism $\Phi\colon T\rightarrow\PP$, that is the weighted blow up of the point $p_4\in\PP$. This, equivalently, can be seen as the map given by the full Riemann-Roch space of the linear system $|\mathcal{O}_T(n(a_4,1))|$, for a large enough $n$. 

Note that, the variety $Y$, corresponding to the birational transform of $X$ under $\Phi$, is a weighted blow up of $p_4\in X$. This map is, however, the Kawamata blow up of this point for a unique choices of $\alpha$, as we see in Lemma~\ref{alpha}. 
\end{enumerate}

In the notation above, suppose $f=x_4^kx_3+g$. The blow up $\Phi\colon T\rightarrow \PP$ is given in coordinates by
\begin{equation}\label{blow-up-map}(u,\x_4,\x_3,\x_2,\x_1,\x_0)\stackrel{\Phi}{\mapsto}(u^\frac{1}{a_4}\x_0:u^\frac{a_4-a}{a_4}\x_1:u^\frac{a}{a_4}\x_2:u^\frac{\alpha}{a_4}\x_3:\x_4)=(x_0:x_1:x_2:x_3:x_4)\end{equation}

Note that the fractional powers do not make this map nonalgebraic; for an argument on this see~\cite{yarek}. Suppose
\[F=\Phi^*(f)\qquad\text{ and }\qquad G=\Phi^*(g)\]
and denote by $m_f$ and $m_g$ the maximal rational numbers for which
\[F=u^{m_f}\overline{f}\quad\text{ and }\quad G=u^{m_g}\overline{g}\quad\text{such that}\quad \overline{f},\overline{g}\in\C[u,\x_0,\dots,\x_4]\]
such that $\overline{f}$ and $\overline{g}$ are polynomials.

\begin{lm}\label{alpha} In the notation above, $\alpha=a_4m_g$ if and only if $Y\rightarrow X$ is the Kawamata blow up of $X$ at $p$.
\end{lm}
\begin{proof} It is easy to compute $-K_T$ and $-K_\PP$ using toric formulae. Then by adjunction, one can compute $-K_Y$ and $-K_X$ and, using the map~(\ref{blow-up-map}), deduce 
\[K_Y-\varphi^*(K_X)=(\frac{\alpha}{a_4}-m_f+\frac{1}{a_4})E\]
Clearly this is the Kawamata blow up if and only if $m_f=\alpha/a_4$. On the other hand 
\[\Phi^*(f)=u^\frac{\alpha}{a_4}\x_4^k\x_3+u^{m_g}\overline{g}\]
Clearly we get the Kawamata blow up if and only if $m_f=m_g=\alpha/a_4$. Note that $m_f<m_g$ implies $E$ is not irreducible.
\end{proof}

\subsection{Fano hypersurfaces}\label{hyper}
In this part, we exclude a singular point on one of the Fano 3-fold hypersurfaces in the complete list of $95$ families of Reid--Iano-Fletcher \cite{fle}.

%\subsubsection{Birational involution on}
We pick  $X_5\subset\PP(1,1,1,1,2)$ as the first in the list. In~\cite[Example~4.5]{CPR} the singular point is excluded by the other method, when $X$ is general. However, the method does not work when the generality is not assumed, as we see below. For the general case we show that the 2-ray game on the blow up of this point produces a birational self-map to $X$, and the 2-ray game on the special case excludes this point immediately using Theorem\,\ref{idea-Fano}. For the exclusion of this point with the other method see \cite[Lemma~5.2.2]{VJ-95}. The behaviour of other quotient singular points on 3-folds in the list of 95 families can be investigated similarly.

The 3-fold $X$ is defined as $(f=0)\subset\PP=\PP(1,1,1,1,2)$, where
\[f=s^2x+g\text{ in which }g=sa_3(x,\dots,t)-b_5(x,\dots,t)\]
for homogeneous polynomials $a_3$ and $b_5$, of degrees indicated by their indices. With respect to the map~(\ref{blow-up-map}) the blow up $\Phi\colon T\rightarrow\PP$ is given by 
\[(u,\s,\t,\z,\y,\x)\stackrel{\Phi}{\mapsto}(u^\frac{\alpha}{2}\x:u^\frac{1}{2}\y:u^\frac{1}{2}\z:u^\frac{1}{2}\t:\s)\]
We have that
\[m_g=\left\{\begin{array}{lcll}
3/2&&\text{if }a_3(0,y,z,t)\not\equiv 0&\qquad\text{General case}\\
5/2&&\text{otherwise}&\qquad\text{Special case}
\end{array}\right.\]
We study these two cases separately, as they actually behave very differently.

\subsubsection{Special case: no link from the blow up}\label{special(1112)}

Using Lemma~\ref{alpha}, $\alpha=5/2$ (therefore $m_g=5/2$), hence the toric space has the Cox ring $R_T=\C[u,\s,\t,\z,\y,\x]$, with irrelevant ideal $I_T=(u,\s)\cap(\x,\dots,\t)$ and the grading

\[\left(\begin{array}{cccccc}
u&\s&\t&\z&\y&\x\\
0&2&1&1&1&1\\
1&1&0&0&0&-2\end{array}\right)\]
The 2-ray game on $T$ is $T\rightarrow\PP$ on one side, and the morphism $\Psi\colon T\rightarrow\PP^\prime=\PP(1,2,2,2,5)$ on the other side, given by the linear system $|\frac{1}{2}\mathcal{O}_T(1,0)|$, that is
\[(u,\s,\t,\z,\y,\x)\stackrel{\Psi}{\mapsto}(\x^\frac{1}{2}u,\y,\z,\t,\x^\frac{5}{2}\s)\]
Note that this map contracts the divisor $\widetilde{E}=(\x=0)$ in $T$ to $\PP^2_{y:z:t}\subset\PP^\prime$. Let us have a closer look at this contraction. The divisor $\widetilde{E}$, as a scheme, is the toric 3-fold defined by 

\[\left(\begin{array}{ccccc}
u&\s&\t&\z&\y\\
0&2&1&1&1\\
1&1&0&0&0\end{array}\right)\]
which is a $\Proj_{\PP^2}(\mathcal{O}_{\PP^2}\oplus\mathcal{O}_{\PP^2}(2)$, hence a $\PP^1$-bundle over $\PP^2$. And $\Psi$ on $\widetilde{E}$ is just the projection to the $\PP^2$, that contracts fibring $\PP^1$. 

On the other hand, $Y\subset T$ is defined by the vanishing of
\[\overline{f}=\s^2\x+u\s\x a^\prime(u\x,\y,\z,\t)-b(u\x,\y,\z,\t)\]
where $a^\prime=\frac{a}{\x}$.
Now we restrict $\Psi$ to $Y$. We have that a divisor $E:(\x=0)\subset Y$, that is $E\colon(b(0,\y,\z,\t)=0)\subset\widetilde{E}$, contracts via this restriction to the quintic curve $(b(0,\y,\z,\t)=0)\subset\PP^2$. Hence $Y$ 2-ray follows $T$. However, by adjunction formula
\[-K_Y=(-K_T-Y)|_Y=\mathcal{O}_Y(1,0)\]
as $Y\in|\mathcal{O}_T(5,0)|$. On the other hand, 
\[\MMob(Y)=\left<(2,1),(1,0)\right>\]
In particular, $-K_Y\in\partial\Mob(Y)$, and not in the interior. By Theorem~\ref{idea-Fano} the singular point cannot be the centre of a maximal singularity.

\subsubsection{General case: birational involution}

Similar to the previous case the toric space has the Cox ring $R_T=\C[u,\s,\t,\z,\y,\x]$, with irrelevant ideal $I_T=(u,\s)\cap(\x,\dots,\t)$ and the grading

\[\left(\begin{array}{cccccc}
u&\s&\t&\z&\y&\x\\
0&2&1&1&1&1\\
1&1&0&0&0&-1\end{array}\right)\]
The 2-ray game on $T$ is $T\rightarrow\PP$ on one side, and the morphism $\Psi\colon T\rightarrow\PP^\prime=\PP(1,1,1,1,3)$ on the other side, given by the linear system $|\mathcal{O}_T(1,0)|$, that is
\[(u,\s,\t,\z,\y,\x)\stackrel{\Psi}{\mapsto}(\y,\z,\t,\x u,\x\s)\]
Note that this map contracts the divisor $\widetilde{E}=(\x=0)$ in $T$, which is exactly as in the previous case.
On the other hand, $Y\subset T$ is defined by the vanishing of
\[\overline{f}=\s^2\x+\s a(u\x,\y,\z,\t)-ub(u\x,\y,\z,\t)\]
Now we restrict $\Psi$ to $Y$. In the first place we see that $E=\widetilde{E}|_Y$ is a hypersurface in $\widetilde{E}$ defined by $(\s a(0,\y,\z,\t)-ub(0,\y,\z,\t)=0)\subset\widetilde{E}$. Looking at the map $\widetilde{E}\rightarrow\PP^2$ restricted to $E$ one can see that $E\rightarrow\PP^2$ is a $1$-to-$1$ map except at the $15$ solutions of $a(0,y,z,t)=b(0,y,z,t)=0)\subset\PP^2$, where $15$ copies of $\PP^1$ are contracted. Hence, on $Y$ we have a small contraction, while on $T$ it was a divisorial contraction. Therefore $Y$ does not 2-ray follow $T$. 

In order to recover the correct toric ambient space, where $Y$ is embedded and it 2-ray follows, we observe that the pre-image of $\PP^2\subset\PP^\prime$ is defined by $u=\s=0$ and also that $f$ vanishes along this. Note that $(u,\s)$ is a component of the irrelevant ideal of $T$, hence we call this copy of $\PP^2$ as ``fake divisor'' on $Y$. To recover the correct toric ambient space we use the method of unpojection, away from the (fake) divisor $(u=\s=0)$, as follows. We introduce a new variable 
\begin{equation}\label{def-equation}r=\frac{\s\x+a(u\x,\y,\z,\t)}{u}=\frac{b(u\x,\y,\z,\t)}{\s}\in\mathcal{O}(3,-1)\end{equation}
Now construct a new toric variety $T^\prime$ with the Cox ring $R_{T^\prime}=\C[u,\s,\t,\z,\y,r,\x]$, irrelevant ideal $I_{T^\prime}=(u,\s)\cap(r,\x,\dots,\t)$ and the grading

\[\left(\begin{array}{ccccccc}
u&\s&\t&\z&\y&r&\x\\
0&2&1&1&1&3&1\\
1&1&0&0&0&-1&-1\end{array}\right)\]
Now define $Y^\prime\subset T^\prime$ as the complete intersection defined by
\[\s\x+a(u\x,\y,\z,\t)-ru=b(u\x,\y,\z,\t)-r\s=0\]
which come from the relations in (\ref{def-equation}). Note that under the projection $\pi_r\colon T^\prime\rightarrow T$ we have that $Y$ and $Y^\prime$ are isomorphic.

The 2-ray game on $T^\prime$ goes as follows. 
\[
\xymatrixcolsep{1.5pc}\xymatrixrowsep{2.3pc}
\xymatrix{
&T^\prime\ar_{\Phi}[ld]\ar^{\Psi}[rd]\ar^{\text{FLIP}}@{-->}[rr]&&T_1\ar_{\Phi_1}[ld]\ar^{\Psi_1}[rd]&\\
\PP&&Z_0&&\PP}\]
where 
\begin{enumerate}[(i)]
\item $\PP=\PP(1,1,1,1,2,3)$,
\item $\Phi\colon T^\prime\rightarrow\PP$ is given by the ray spanned by $\frac{1}{2}D_\s$ and in coordinates is
\[(u,\s,\t,\z,\y,r,\x)\stackrel{\Phi}{\mapsto}(u^\frac{1}{2}\y:u^\frac{1}{2}\z:u^\frac{1}{2}\t:u^\frac{3}{2}\x:\s:u^\frac{5}{2}r)\]
\item $\Psi_1\colon T_1\rightarrow\PP$ is given by the ray spanned by $\frac{1}{2}D_r$ and in coordinates is
\[(u,\s,\t,\z,\y,r,\x)\stackrel{\Psi}{\mapsto}(x^\frac{1}{2}\y:x^\frac{1}{2}\z:x^\frac{1}{2}\t:x^\frac{3}{2}u:r:x^\frac{5}{2}\s)\]\item the map $T^\prime\rightarrow T_1$ is a flop of type $(1,1,-1,-1)$ over $\PP^2$, in the sense that
\begin{enumerate}[(1)]
\item $\Psi$ contracts the locus (of codimension $2$) defined by $(r=\x=0)$, which is the toric variety with Cox ring
\[\left(\begin{array}{ccccc}
u&\s&\t&\z&\y\\
0&2&1&1&1\\
1&1&0&0&0\end{array}\right)\]
Clearly $\Psi$, on this locus, contracts fibring $\PP^1$s to the base $\PP^2$.
\item $\Phi_1$, on the other hand, is very similar and contracts the locus $(u=\s=0)$ to the same $\PP^2$ with $\PP^1$ fibres above each point. \end{enumerate}
In other words the composition map $T^\prime\dashrightarrow T_1$ flops these $\PP^1$s at each fibre above the base $\PP^2$.
\end{enumerate}

Now we restrict all these maps to $Y^\prime$, and its images. We first observe that $X$, the initial variety, can be embedded in $\PP$ via the embedding of $\PP(1^4,2)\subset\PP$. In particular, $X$ is a complete intersection of two hypersurfaces as follows
\[(r-sx-a(u,y,z,t)=b(u,y,z,t)-rs=0)\subset\PP\]
It is now easy to check that for the restriction maps we get the following diagram

\[
\xymatrixcolsep{1.5pc}\xymatrixrowsep{2.3pc}
\xymatrix{
&Y^\prime\ar_{\varphi}[ld]\ar^{\psi}[rd]\ar^{\text{flop}}@{-->}[rr]&&Y_1\ar_{\varphi_1}[ld]\ar^{\psi_1}[rd]&\\
X&&X_0&&X_1}\]
where $\varphi\colon Y^\prime\rightarrow X$ is the Kawamata blow up. The map $Y^\prime\dashrightarrow Y_1$ is the flop of the $15$ copies of $\PP^1$ that correspond to the solutions of $(a(0,y,z,t)=b(0,y,z,t)=0)\subset\PP^2$. 

Now, $Y_1$ is defined as the vanishing of 
\[\s\x+a(u\x,\y,\z,\t)-ru=b(u\x,\y,\z,\t)-r\s=0\]
in the toric variety $T_1$, where the Cox ring of $T_1$ is the same as that of $T$ except the irrelevant ideal has changed to $I_{T_1}=(r,\x)\cap(u,\s,\y,\z,\t)$.

It can also be checked that the last map $\psi_1\colon Y_1\rightarrow X_1$ contracts the divisor $E_1\subset Y_1$, that is given by ($\x=0)$, and is isomorphic to $\PP^2$. The image of the contraction is the $p_r$ point in $X_1\subset\PP$, where $X_1$ is defined by 
\[s+a(u,y,z,t)-ru=b(u,y,z,t)-rs=0\]
Obviously the variable $s$ can be eliminated to give $X_1\subset\PP(1,1,1,1,2)$ as a hypersurface defined by
\[r^2u-ra(u,y,z,t)-b(u,y,z,t)=0\]
Clearly $X$ is isomorphic to $X_1$, not via this construction.

\begin{rmk} This is the case of {\it quadratic involution}, as it is called in \cite[\S4.4]{CPR}. Another type of involution is the {\it elliptic involution} as in \cite[\S4.10]{CPR}. One can see that the elliptic involutions can be constructed very similarly with a few more unprojections.\end{rmk}

%\subsubsection{Exclusion of the $1/2(1,1,1)$ singular point on $X_7\subset\PP(1,1,1,2,3)$}

%\subsubsection{The ``amazing example'' $X_{18}\subset\PP(1,3,4,5,6)$ and Type~\III\, unprojection}

\section{Complete intersection Fano 3-folds}\label{c.i}
The question of birational rigidity for quasi-smooth Fano hypersurfaces of index one was completed recently in \cite{VJ-95}. A similar study was carried on by Okada for Fano varieties in codimension two. In this section, we complete his work by removing the generality assumptions he imposed.

\subsection{Okada's list}
There are 85 families of Fano 3-fold wighted complete intersections in codimension 2 \cite[\S16.7]{fle}. It was proved in \cite{pukh-c.i.} that a general member in one of these families, that is the smooth complete intersection of a quadric and a cubic in $\PP^5$, is birationally rigid. Okada in \cite{okada} systematically studied birational rigidity and birational non-rigidity for (general members in) all other 84 families. In particular, he showed birational rigidity for 18 of these families (in addition to $X_{2,3}\subset\PP^5$), and for general members in others the pliability set includes at least two elements. Within those 18 families he assumes some generality conditions for 8 cases. Here we study the special cases, not considered by Okada, for all cases where the ambient space is singular. We prove that the generality assumptions can be lifted. The families with the generality conditions are presented in Table~\ref{okada-table}.

We denote a Fano variety in codimension two by 
\[X_{d_1,d_2}\subset\PP(a_0,a_1,a_2,a_3,a_4,a_5)\]
where $X$ is defined as the intersection of $\{f=0\}\cap\{g=0\}$, where $f$ and $g$ are homogenous polynomials of degree, respectively, $d_1$ and $d_2$, with respect to the grading on the weighted projective space, and $d_1\leq d_2$ and $a_0\leq a_1\leq\cdots\leq a_5$. We denote the variables on the weighted projective $5$-fold by $x,y,z,t,s,w$, in that order.

\begin{table}[h]\label{okada-table}\centering 
\begin{tabular}{|c|c|c|c|}
\hline
No.&$X_{d_1,d_2}\subset$WPS&Singular point&Generality assumption\\
\hline
$20$&$X_{6,8}\subset\PP(1,2,2,3,3,4)$&$6\times\frac{1}{2}(1,1,1)$&$\Bs|I_{p,X}(-2K_X)|$ is irreducible $\forall p\!:\!\frac{1}{2}(1,1,1)$\\
\hline
31&$X_{8,10}\subset\PP(1,2,3,4,4,5)$&$\frac{1}{3}(1,1,2)$&$wz\in f$ and $z^2s, z^2t\in g$\\
\hline
37&$X_{8,12}\subset\PP(1,2,3,4,5,6)$&$4\times\frac{1}{2}(1,1,1)$&$s^2\in f$\\
\hline
&$X_{8,12}\subset\PP(1,2,3,4,5,6)$&$2\times\frac{1}{3}(1,1,2)$&$s^2\in f$ and $s^3\in g$\\
\hline
47&$X_{10,12}\subset\PP(1,3,4,4,5,6)$&$2\times\frac{1}{3}(1,1,2)$&the curve $C_{1,4,5,6}\not\subset X$\\
\hline
51&$X_{10,14}\subset\PP(1,2,4,5,6,7)$&$\frac{1}{4}(1,1,3)$&the curve $(x=y=0)\cap X$ is irreducible\\
\hline
59&$X_{12,14}\subset\PP(1,4,4,5,6,7)$&$3\times\frac{1}{4}(1,1,3)$&$C_{1,4,6,7}\not\subset X$\\
\hline
64&$X_{12,16}\subset\PP(1,2,5,6,7,8)$&$\frac{1}{5}(1,2,3)$&$z^2s\in g$\\
\hline
71&$X_{14,16}\subset\PP(1,4,5,6,7,8)$&$2\times\frac{1}{4}(1,1,3)$&$C_{1,7,8,10}\not\subset X$\\
\hline
&$X_{14,16}\subset\PP(1,4,5,6,7,8)$&$\frac{1}{5}(1,2,3)$&$z^2s\in g$\\
\hline
\end{tabular}
\vspace{0.15cm}
\caption{\small rigid families with generality assumptions in Okada's list. See \cite[Theorem~1.2]{okada} for details.} 
\end{table} 

\begin{rmk} All varieties considered are complete intersections and only admit cyclic quotient singularities, inherited from their ambient weighted projective space. Then $\Q$-factoriality follows from a theorem of Grothendieck, see \cite[Lemma 3.5]{CPR}, and the fact that the Picard group is isomorphic to $\Z$ follows from Lefschetz' hyperplane section theorem.
\end{rmk}

\subsection{Global Kawamata blow up of Fano complete intersections}
In each case, we start by a triple $(\PP,X,p)$, coming from the Table~\ref{okada-table}, so that $X$ is the complete intersection in the weighted projective space $\PP$ and $p\in X$ is a quotient terminal singularity. Similar to Subsection~\ref{toric-method}, we construct a toric variety $T$ together with a morphism $\Phi:T\rightarrow\PP$ such that the restriction map $\varphi=\Phi|_Y:Y=\Phi^{-1}_*(X)\rightarrow X$ is the Kawamata blow up of $X$ at the point $p$. The main difference to that case is that instead of one equation we have two! But the multiplicities can be computed exactly in the same way and it can be checked that, just as in the hypersurface case,  the Kawamata blow up is obtained with the correct blow up {\it weights for the tangent variables}.

\subsection{Proof of Theorem\,\ref{rigidity-Fano}}

\subsubsection{Exclusion via the full Cox ring}\label{full-Cox}

 \paragraph*{\bf Family No.\,$64$} $\quad\frac{1}{5}(1,2,3)\in X_{12,16}\subset\PP(1,2,5,6,7,8)$\\
 The generality assumption imposed by Okada is that the monomial $z^2s$ appears in $g$ with nonzero coefficient.
 
 The quasi-smoothness of $X$ implies that $s^2, tz\in f$ and also $w^2,t^2y,z^3x\in g$, in other words
 \[f=s^2+ tz+\cdots\quad\text{and}\quad g=w^2+t^2y+z^3x+\cdots\]
The tangent monomials to the non-vanishing variable $z$, at the singular point, are $t$ and $x$. This implies that the blow up must locally be of type $u^\frac{1}{5}s,u^\frac{2}{5}y,u^\frac{3}{5}w$. It is now easy to check that $m_f=2/5$ and $m_g=6/5$. Therefore the Cox ring of $T$ is $\C[u,\z,\s,\t,\w,\y,\x]$ with weights
\[\left(\begin{array}{ccccccc}
u&\z&\s&\t&\w&\y&\x\\
0&5&6&7&8&2&1\\
-5&0&1&2&3&2&1\end{array}\right)\]
which normalises (by subtracting the first row from the second row and then dividing by $-5$) to
\[\left(\begin{array}{ccccccc}
u&\z&\s&\t&\w&\y&\x\\
0&5&6&7&8&2&1\\
1&1&1&1&1&0&-1\end{array}\right)\]
One can check that as in the examples in the previous section, $Y$ 2-ray follows $T$. In particular the maps, at the level of 3-folds,  are
\[
\xymatrixcolsep{1.5pc}\xymatrixrowsep{2.3pc}
\xymatrix{
&Y\ar_{\varphi}[ld]\ar^{\text{isomorphism}}_{\s^2\in\overline{f}}@{->}[rr]&&Y_1\ar^{\text{antiflip}}_{(-7,-1,1,8)}@{-->}[rr]&&Y_2\ar^{\text{isomorphism}}_{\w^2\in\overline{g}}@{->}[rr]&&Y_3\ar^{\psi}[rd]&\\
X&&&&&&&&X_1}\]
In particular, the mobile cone of $Y$ is generated by $(5,1)$ and $(2,0)$. On the other hand, using adjunction formula $-K_Y\in|\mathcal{O}_Y(1,0)|$, hence lies in the boundary of the mobile cone and therefore this blow up does not initiate a Sakisov link. In fact it gives a bad link.\\

 \paragraph*{\bf Family No.\,$47$} $\quad\frac{1}{3}(1,1,2)\in X_{10,12}\subset\PP(1,3,4,4,5,6)$\\
The generality assumption of Okada in this case is equivalent to $f$ not having terms $y^2z$ and $y^2s$, and also $g$ not having the terms $s^3$ and $z^3$. Combining the quasi-smoothness conditions at the vertex points, together with these assumptions we have that
\[f=t^2+y^3x+\cdots\quad\text{and}\quad g=w^2+y^2w+zs\ell(z,s)+\cdots\]
where $\ell$ is a linear form in $z$ and $s$. Without loss of generality, we assume the singular point is the $p_z$. The arguments for the two other points are completely symmetric. One can check that $m_f=4/3$ and $m_g=1$. Hence the weights in the Cox ring of the blow up toric variety are
\[\left(\begin{array}{ccccccc}
u&\y&\z&\s&\t&\w&\x\\
0&3&4&4&5&6&1\\
1&1&1&1&1&1&-1\end{array}\right)\]
The 2-ray game (following the ambient space) is
\[
\xymatrixcolsep{1.5pc}\xymatrixrowsep{2.3pc}
\xymatrix{
&Y\ar_{\varphi}[ld]\ar^{\text{antiflip}}_{3\times(-3,-1,1,5)}@{-->}[rr]&&Y_1\ar^{\text{isomorphism}}_{\t^2\in\overline{f}}@{->}[rr]&&Y_2\ar^{\psi}[rd]&\\
X&&&&&&X_1}\]
On the other hand $-K_Y\in|\mathcal{O}_Y(1,0)|$, which lies outside of the mobile cone. Hence the exclusion follows by mobility obstruction.\\

 \paragraph*{\bf Family No.\,$71$} $\quad\frac{1}{5}(1,2,3)\in X_{14,16}\subset\PP(1,4,5,6,7,8)$\\
Assume that the generality condition does not hold, i.e., $z^2s$ does not appear in $g$. Using the quasi-smoothness conditions we have
\[f=t^2+sw+z^2y+\cdots\quad\text{and}\quad g=w^2+s^2y+z^3x+\cdots\]
Using the assumptions, one can check that $m_f=4/5$ and $m_g=6/5$. The matrix of weights on the blow up variety is

\[\left(\begin{array}{ccccccc}
u&\z&\s&\t&\w&\y&\x\\
0&5&6&7&8&4&1\\
1&1&1&1&1&0&-1\end{array}\right)\]
The 2-ray game (following the ambient space) is
\[
\xymatrixcolsep{1.5pc}\xymatrixrowsep{2.3pc}
\xymatrix{
&Y\ar_{\varphi}[ld]\ar^{\text{antiflip}}_{(-6,-1,1,7)}@{-->}[rr]&&Y_1\ar^{\text{isomorphism}}_{\t^2\in\overline{f}}@{->}[rr]&&Y_2\ar^{\text{isomorphism}}_{\w^2\in\overline{g}}@{->}[rr]&&Y_3\ar^{\psi}[rd]&\\
X&&&&&&&&X_1}\]

Also $-K_Y\in|\mathcal{O}_Y(1,0)|$ which lies on the boundary of the mobile cone. Hence it is excluded by a bad link.\\

\paragraph*{\bf Family No.\,$31$} $\quad\frac{1}{3}(1,1,2)\in X_{8,10}\subset\PP(1,2,3,4,4,5)$\\
This case is more delicate, as it splits into several sub-cases. First, recall that Okada's generality conditions imply $zw\in f$ and $z^2s,z^2t\in g$. Suppose, without loss of generality, that $z^2t\notin g$ and $zw\notin f$. The quasi-smoothness at the point $p_z$ implies that $z^2y\in f$ and one of $z^2s$ or $z^3x$ appear in $g$ with non-zero coefficient.

Suppose $z^2s\notin g$ and $z^3x\in g$. It is then easy to see, similar to the cases above, that $m_f=2/3$ and $m_g=4/3$. The Cox ring is
\[\left(\begin{array}{ccccccc}
u&\z&\s&\t&\w&\y&\x\\
0&3&4&4&5&2&1\\
1&1&1&1&1&0&-1\end{array}\right)\]
Then $Y$ 2-ray follows $T$ and the game is a $(-4,-1,1,5)$ anti-flip followed by a $K$-trivial contraction. This is because $-K_Y$ is on the boundary of the mobile cone, providing a bad link.

Now suppose $z^2s\in g$. With a change of coordinates we can kill the terms $x^6y$, $z^2x^2$ and $x^4t$ in $f$ as well as $z^3x$ in $g$. Moreover, we can assume that the $1/4(1,1,3)$ singularities are at the vertex points $p_s$ and $p_t$, and also $p_y\in X$. In particular we have
\[f= z^2y+ts+\cdots\quad\text{and}\quad g=w^2+z^2s+s^2y+t^2y+\cdots\]

It is now easy to check that $m_g=4/3$, and $m_f=2/3$ or $m_f=5/3$. In fact $m_f=2/3$ if and only if the monomial $zxt$ appears in $f$. We call it  Family No.\,$31^*$ and will treat it in Subsection\,\ref{elliptic}. Suppose $m_f=5/3$, which indicates that the Cox ring is
\[\left(\begin{array}{ccccccc}
u&\z&\t&\w&\x&\s&\y\\
0&3&4&5&1&4&2\\
1&1&1&1&0&0&-1\end{array}\right)\]
Again $Y$ 2-ray follows $T$, and $\MMob(Y)=\left<(3,1),(1,0)\right>$. The fact that $-K_Y\in\mathcal{O}_Y(1,0)$, using adjunction formula, implies that it lies in the boundary of the mobile cone.\\

\paragraph*{\bf Family No.\,$20$} $\quad\frac{1}{2}(1,1,1)\in X_{6,8}\subset\PP(1,2,2,3,3,4)$\\
Let use assume first, without loss of generality, that the two $1/3(1,1,2)$ points correspond to $p_s$ and $p_t$. This is to say that $f(0,0,0,s,t,0)=st$. Quasi-smoothness condition on $X$ requires polynomials of the form $s^2\ell_1(y,z)$ and $t^2\ell_2(y,z)$ to appear in $g$ with non-zero coefficients, where $\ell_i$ are linear forms in $y$ and $z$.

Without loss of generality we can assume the point to be excluded is the vertex point $p_z$. For this point $\Bs|I_{p,X}(-2K_X)|$ corresponds to the scheme \[\{\widetilde{f}=\widetilde{g}=0\}\subset\PP(2,3,3,4)\]
where $\widetilde{f}=f(0,0,z,s,t,w)$ and $\widetilde{g}=g(0,0,z,s,t,w)$. The generality assumption in this case requires this scheme to be irreducible. However,  note that $\widetilde{f}$ consists of monomials $st$ and $zw$, while $\widetilde{g}$ contains (possibly) the monomials $w^2, s^2z,t^2z,tsz, z^2w$. On the other hand $\widetilde{f}$ must contain $st$, for $X$ to be quasi-smooth. This base locus fails to be irreducible if either  $zw\in\widetilde{f}$ and none of the monomials $s^2z$, $t^2z$, $tsz$ appear in $\widetilde{g}$ (Case\,I), or $zw\notin\widetilde{f}$ (Case\,II).

Case\,I. Clearly $\ell_1$ and $\ell_2$ are non-trivial linear forms in $y$ and we have
\[f=st+zw+\cdots\quad\text{and}\quad g=w^2+s^2y+t^2y+z^3y+\cdots\]

The blow up locally is given by $u^\frac{1}{2}x,u^\frac{1}{2}s,u^\frac{1}{2}t$. Now $st\in f$ implies that $m_f=1$. On the other hand $m_g=1$ if and only if one of the monomials $z^3x^2$, $z^2xs$, $z^2xt$ or $z^2w$ appear in $g$ with non-zero coefficient. Note that we do not need to consider the monomial $z^3x^2$, as this can be killed using the term $z^3y$.

If $m_g=1$, then the Cox ring of the toric variety $T$ is 
\[\left(\begin{array}{ccccccc}
u&\z&\s&\t&\w&\x&\y\\
0&2&3&3&4&1&2\\
1&1&1&1&1&0&0\end{array}\right)\]
and $Y$ is defined by the vanishing of
\[\overline{f}=\s\t+\z\w+\cdots\]
\[\overline{g}=\z^2\left(\alpha \w+\beta \x\s+\gamma \x\t\right)+u\left(\w^2+\s^2\y+\t^2\y+\z^3\y+\cdots\right)\]
where at least one of $\alpha,\beta,\gamma$ is non-zero. The need for unprojection is clear now! We have a new variable 
\[r=\frac{\alpha \w+\beta \x\s+\gamma \x\t}{-u}=\frac{\w^2+\s^2\y+\t^2\y+\z^3\y+\cdots}{\z^2}\]
The new toric variety $T'$ has the Cox ring
\[\left(\begin{array}{cccccccc}
u&\z&\s&\t&\w&\x&\y&r\\
0&2&3&3&4&1&2&4\\
1&1&1&1&1&0&0&0\end{array}\right)\]
and $Y$ is defined by the equations
\[\left\{\begin{array}{l}
ru+\alpha \w+\beta \x\s+\gamma \x\t=0\\
r\z^2=\w^2+\s^2\y+\t^2\y+\z^3\y+\cdots\\
\s\t+\z\w+\cdots=0
\end{array}
\right.\]
It is now easy to check that $Y$ 2-ray follows $T'$ and the game ends with an elliptic fibration over $\PP(1,2,4)$, a bad link.

Now suppose $m_g=2$. Then the Cox ring of $T$ is given by
\[\left(\begin{array}{ccccccc}
u&\z&\s&\t&\w&\x&\y\\
0&2&3&3&4&1&2\\
1&1&1&1&1&0&-1\end{array}\right)\]
and $Y$ is defined by
\[\overline{f}=\s\t+\z\w+\cdots\]
\[\overline{g}=\w^2+u\s^2\y+u\t^2\y+\z^3\y+\cdots\]
It is now easy to check that $Y$ 2-ray follows $T$ and the game ends in a bad link. This is because $(1,0)$ spans a boundary of the mobile cone and $-K_Y$ lies exactly on this boundary.

Case\,II. Suppose $zw\notin f$. The quasi-smoothness condition implies that $z^2y\in f$, and hence $z^2w\in g$, and we have
\[f=st+z^2y+\cdots\quad\text{and}\quad g=w^2+z^2w+s^2\ell_1(y,z)+t^2\ell_2(y,z)+\cdots\]
As the blow up is locally given by $u^\frac{1}{2}x,u^\frac{1}{2}s,u^\frac{1}{2}t$, we have that $m_f=1$ due to appearance of $st$ in $f$. On the other hand $m_g=1$ if and only if either $z\in\ell_i$, for some $i$, or (at least) one of the monomials
\[z^3x^2,z^2xs,z^2xt,z^3y, zst\]
appears in $g$ with non-zero coefficient. However, all these monomials can be killed using the term $z^2w$, and a change of coordinates. Hence we only need to consider the case where
\[\alpha s^2z+\beta t^2z+\gamma stz\in g\]
and at least one of the coefficients $\alpha$, $\beta$ or $\gamma$ is non-zero. Then the Cox ring of the toric ambient space $T$ is
\[\left(\begin{array}{ccccccc}
u&\z&\s&\t&\w&\x&\y\\
0&2&3&3&4&1&2\\
1&1&1&1&1&0&0\end{array}\right)\]
and $Y$ is defined by
\[\overline{f}=\s\t+\z^2\y+\cdots\quad\text{and}\quad \overline{g}=\z( \z\w+\alpha\s^2+\beta\t^2+\gamma\s\t)+u(\w^2+\cdots)\]
After unprojection (using $\overline{g}$) we get a toric space $T'$ with cox ring
\[\left(\begin{array}{cccccccc}
u&\z&\s&\t&\w&r&\x&\y\\
0&2&3&3&4&6&1&2\\
1&1&1&1&1&1&0&0\end{array}\right)\]
and $Y$ is defined by
\[\left\{\begin{array}{l}
r\z+\w^2+\cdots=0\\
ru+\z\w+\alpha\s^2+\beta\t^2+\gamma\s\t=0\\
\s\t+\z^2\y+\cdots=0
\end{array}
\right.\]
It can be checked that if $\alpha\neq 0$ or $\beta\neq 0$, then $Y$ 2-ray follows $T'$ and we obtain a $K3$-fibration over $\PP(1,2)$, hence a bad link.

On the other hand, if $\alpha=\beta=0$, then the equations defining $Y$ are
\[\left\{\begin{array}{l}
r\z+\w^2+\s^2\y+\t^2\y+\cdots=0\\
ru+\z\w+\s\t=0\\
\s\t+\z^2\y+\cdots=0
\end{array}
\right.\]
However, substituting $\s\t$ from the second equation to the third shows that (the new) third equation is in the ideal $(u,\z$, hence we must unproject this divisor. Once this is done we get a new variable
\[\eta=\frac{r+\cdots}{\z}=\frac{\w+\z\y+\cdots}{u}\]
which has weight $(4,0)$. After rewriting the equations from these ratios, we can eliminate variables $r$ and $\w$ to get a new toric variety $T''$ with Cox ring
\[\left(\begin{array}{ccccccc}
u&\z&\s&\t&\x&\y&\eta\\
0&2&3&3&1&2&4\\
1&1&1&1&0&0&0\end{array}\right)\]
and the equations defining $Y$ become
\[\left\{\begin{array}{l}
\s\t+\z(u\eta+\z\y+\cdots)+u(\eta\z+\cdots)=0\\
\z^2\eta+\z(\cdots)+(u\eta+\z\y+\cdots)^2+\s^2\y+\t^2\y+\cdots=0
\end{array}
\right.\]
Now, $Y$ 2-ray follows $T''$ and the game end with an elliptic fibration over $\PP(1,2,4)$.

Now suppose that $m_g=2$. In this case the Cox ring of $T$ is
\[\left(\begin{array}{ccccccc}
u&\z&\s&\t&\w&\x&\y\\
0&2&3&3&4&1&2\\
1&1&1&1&1&0&-1\end{array}\right)\]
and $Y$ is defined by 
\[\overline{f}=\s\t+\z^2\y+\cdots\quad\text{and}\quad \overline{g}=u\w^2+z\z^2\w+\s^2\ell_1+\t^2\ell_2+\cdots\]
We have one unprojection to introduce the new variable $r$ with weight $(6,1)$. Then the new toric space is $T'$ with Cox ring
\[\left(\begin{array}{cccccccc}
u&\z&\s&\t&\w&r&\x&\y\\
0&2&3&3&4&6&1&2\\
1&1&1&1&1&1&0&-1\end{array}\right)\]
and the equations (for $Y$) are
\[\left\{\begin{array}{l}
r\z+\w^2+\cdots=0\\
ru+\s^2+\t^2+\z\w=0\\
\s\t+\z^2\y+\cdots=0\end{array}
\right.\]
Now $Y$ 2-ray follows $T'$ and the game ends with a $K3$-fibration (a bad link).\\

\paragraph*{\bf Family No.\,$37$}$\quad\frac{1}{3}(1,1,2)\in X_{8,12}\subset\PP(1,2,3,4,5,6)$\\
Okada's generality assumption here are $s^2\in f$ and $s^3\in g$. Note that it is required, for $X$ to have terminal singularities that one of these two has to happen. Hence we divide the special case into two sub-cases, i.e.\,when $s^2\in f$ and $s^3\notin g$ and vice versa.

Suppose $s^2\notin f$ and $s^3\in g$. Note that quasi-smoothness of $X$ imply that $tz\in f$ and $w^2,t^2y,z^2w\in g$. Without loss of generality, we can assume that one of the $1/2(1,1,1)$ points is $p_y$, which requires that $y^2s\in f$ or $yw\in f$. But if $wy\notin f$, an easy analysis shows that the other $1/2(1,1,1)$ point is not a quotient singularity, hence $X$ is not quasi-smooth. Therefore we assume that 
\[f=tz+yw+\cdots\quad\text{and}\quad g=s^3+w^2+t^2y+z^2w+y^4s+\cdots\]
On the other hand, note that any of the monomials $z^2x^2,zxs,z^2y$ can be killed, in $f$, using the term $tz$.

The blow up is locally given by $u^\frac{1}{3}x,u^\frac{1}{3}s,u^\frac{2}{3}y$. Hence, $m_g=1$. This together with the term analysis above show that $m_f=5/3$. Now the Cox ring of the toric space is given by
\[\left(\begin{array}{ccccccc}
u&\z&\s&\w&\x&\y&\t\\
0&3&4&6&1&2&5\\
1&1&1&1&0&0&0\end{array}\right)\]
and $Y$ is given by the vanishing of
\[\overline{f}=\t\z+\y\w+\cdots\quad\text{and}\quad \overline{g}=\s^3+u\w^2+u\t^2\y+u^3\z^2\w+u^2\y^4\s+\cdots\]
Now, $Y$ 2-ray follows $T$, and the game ends with an elliptic fibration (a bad link).

Now suppose that $s^2\in f$ and $s^3\notin g$. We have that
\[f=tz+s^2+\cdots\quad\text{and}\quad g=w^2+t^2y+z^2w+\cdots\]
which immediately gives $m_f=2/3$. On the other hand $m_g=1$ if and only if at least one of the monomials below appears in $g$ with non-zero coefficient.
\[zxs^2,zst,z^3x^3,z^2x^2s,z^3xy,z^2sy,z^2xt\]
Note that any of these monomials that is divisible by $z^2$ can be killed using the term $z^2w$. On the other hand, if $zst\in g$, then we can substitute $zt$ from $f$ into $g$, so that $s^3$ appears in $g$. But this is treated by Okada. Hence, we assume that the only term in $g$ that makes $m_g=1$ is $zxs^2$. In this case the Cox ring of the toric space $T$ is
\[\left(\begin{array}{ccccccc}
u&\z&\s&\t&\w&\x&\y\\
0&3&4&5&6&1&2\\
1&1&1&1&1&0&0\end{array}\right)\]
and $Y$ is given by
\[\overline{f}=\t\z+\s^2+\cdots\quad\text{and}\quad \overline{g}=\z^2\w+\z\x\s^2+u\w^2+u\t^2\y+\cdots\]
Note that we can substitute $\s^2$ from the first equation into the second to observe that 
\[\overline{g}=\z^2(\w+\x\t+\cdots)+u(\w^2+\t^2\y+\cdots)\]
Now we can unproject by introducing the new variable $r$ as follows:
\[r=\frac{w+\x\t+\cdots}{u}=\frac{\w^2+\t^2\y+\cdots}{-z^2}\]
that has weight $(6,0)$, and writing down the equation from the first ratio eliminates $\w$. So $Y$ is now embedded in a toric variety $T'$ with Cox ring
\[\left(\begin{array}{ccccccc}
u&\z&\s&\t&\x&\y&r\\
0&3&4&5&1&2&6\\
1&1&1&1&0&0&0\end{array}\right)\]
and the equations are
\[\overline{f}=\t\z+\s^2+\cdots\quad\text{and}\quad \overline{g'}=r\z^2+\t^2\y+(ru-\x\t+\cdots)^2+\cdots\]
In this configuration, $Y$ 2-ray follows $T'$, and the game ends with an elliptic fibration (a bad link).\\

 \paragraph*{\bf Family No.\,$71$}$\quad\frac{1}{4}(1,1,3)\in X_{14,16}\subset\PP(1,4,5,6,7,8)$\\
 We do not use Okada's condition for this case. Without loss of generality, assume that the singular point, $1/4(1,1,3)$, to be excluded, is $p_y$. Quasi-smoothness of $X$ forces $f$ and $g$ to be of the form
 \[f=t^2+sw+z^2y+y^2s+\cdots\quad\text{and}\quad g=w^2+s^2y+y^2w+\cdots\]
 The blw up is locally given by $u^\frac{1}{4}x,u^\frac{1}{4}z,u^\frac{3}{4}t$, which immediately reads $m_f=1/2$ and $m_g=1$. In particular, the Cox ring of the toric ambient space $T$ is
  \[\left(\begin{array}{ccccccc}
u&\y&\z&\s&\t&\w&\x\\
0&4&5&6&7&8&1\\
1&1&1&1&1&1&0\end{array}\right)\]
The 3-fold $Y$ is the vanishing of
 \[\overline{f}=\y(\z^2+\y\s)+u(\t^2+\s\w+\cdots)\quad\text{and}\quad \overline{g}=u\w^2+\s^2\y+\y^2\w+\cdots\]
Unprojecting the divisor $\{u=\y=0\}$ using $\overline{f}$, we introduce a new variable $r$ with weight $(10,1)$, and we get 3 equations
\[\left\{\begin{array}{l}
\z^2+ru+\y\s=0\\
u(\w^2+\cdots)+\y(\s^2+\y\w+\cdots)+\z^2(\s+\z\x)=0\\
r\y=\t^2+\s\w+\cdots
\end{array}
\right.\]
We can substitute $\z^2$ from the first equation into the second (i.e.\,changing the generators of the ideal of $Y$) to see that one equation is in the ideal $(u,\y)$, hence we must unproject again! This gives a new variable $\eta$ with weight $(12,1)$ so that the new toric ambient space $T''$ has the Cox ring
 \[\left(\begin{array}{ccccccccc}
u&\y&\z&\s&\t&\w&r&\eta&\x\\
0&4&5&6&7&8&10&12&1\\
1&1&1&1&1&1&1&1&0\end{array}\right)\]
 and $Y$ is defined by
 \[\left\{\begin{array}{l}
\z^2+ru+\y\s=0\\
r\y=\t^2+\s\w+\cdots\\
\y\eta=\w^2+r\s+r\z\x\\
u\eta=\s^2+\y\w+\z\s\x+\cdots
\end{array}
\right.\]
Note that because of the presence of the monomials $\z^2,\s^2,\t^2,\w^2$ in these polynomials, the first 4 steps of the 2-ray game of $T''$ restricted to $Y$ are isomorphism. For the last step, setting $r=1$ shows that the FLIP on the toric space restricts to the inverse of the flip $(2,1,-5,-3,-2;-8)$. This is not a terminal flip (see\,\cite[Theorem\,8]{gavin}), hence the game fails by Theorem\,\ref{idea-Fano}. However, it is still true that $Y$ 2-ray follows $T''$, and we know that $-K_Y\in|\mathcal{O}_Y(1,0)|$, which lies outside of the mobile cone of $Y$.\\
 
 \paragraph*{\bf Family No.\,$59$}$\quad\frac{1}{4}(1,1,3))\in X_{12,14}\subset\PP(1,4,4,5,6,7)$\\
 Again, we do not use Okada's condition. After a change of coordinates, we can assume that $f(0,y,z,0,0,0)=yz(y+z)$. Suppose the $1/4(1,1,3)$ point to be excluded is $p_z$. Quasi-smoothness of $X$ implies
 \[f=t^2+yz(y+z)+sw+\cdots\quad\text{and}\quad g=w^2+y^2t+z^2t+s^2\ell(y,z)+\cdots\]
 where $\ell$ is a linear form in $y$ and $z$. The blow up is locally given by $u^\frac{1}{4}x,u^\frac{1}{4}s,u^\frac{3}{4}w$, and immediately we get $m_f=1$ as $sw\in f$. On the other hand, $m_g=1/2$ if and only if at least one of the monomials $z^3x^2$, $z^2xs$ or $zs^2$ appears in $g$ with non-zero coefficient. Note that the first two monomials can be killed using the term $z^2t$. Hence we have two subcases, either $m_g=1/2$ or $m_g=3/2$ (when $zs^2\notin g$).
 
 Suppose that $m_g=3/2$. The Cox ring of the toric space $T$ is given by 
 \[\left(\begin{array}{ccccccc}
u&\z&\s&\w&\x&\y&\t\\
0&4&5&7&1&4&6\\
1&1&1&1&0&0&0\end{array}\right)\]
 and $Y$ is defined by the vanishing of
  \[\overline{f}=u\t^2+\y\z(u\y+\z)+\s\w+\cdots\quad\text{and}\quad \overline{g}=\w^2+u^2\y^2\t+\z^2\t+\cdots\]
It is easy to check that $Y$ 2-ray follows $T$ and the game ends with an elliptic fibration (a bad link).

Now, suppose $m_g=1/2$. The Cox ring of $T$ is the same as the previous case except that the variable $\t$ has weight $(6,1)$. The 3-fold $Y$ is defined by
  \[\overline{f}=\t^2+\y\z(u\y+\z)+\s\w+\cdots\quad\text{and}\quad \overline{g}=\z(\s^2+\z\t)+u(\w^2+u\y^2+\cdots)\]
  Clearly, we need to unproject using $\overline{g}$, which gives a new toric variety $T'$ with Cox ring
  \[\left(\begin{array}{cccccccc}
u&\z&\s&\t&\w&r&\x&\y\\
0&4&5&6&7&10&1&4\\
1&1&1&1&1&1&0&0\end{array}\right)\]
and $Y$ is defined by
\[\left\{\begin{array}{l}
ru+\s^2+\z\t=0\\
r\z=\w^2+u\y^2\t+\cdots\\
\z^2\y+\s\w+\t^2+\cdots=0
\end{array}
\right.\]
As before, it can be checked that $Y$ 2-ray follows $T'$ and the game ends with a $K3$-fibration.\\ 

 \paragraph*{\bf Family No.\,$51$}$\quad\frac{1}{4}(1,1,3)\in X_{10,14}\subset\PP(1,2,4,5,6,7)$\\
 We assume that the point to be excluded is $p_z$. It is easy to verify that Okada's condition in this case is to say at least one of the monomials $z^2t$ or $s^2z$ appears in $g$ with non-zero coefficient. Hence, we assume that none of them is in $g$. Quasi-smoothness of $X$ implies that
 \[f=s^2+tz+\cdots\quad\text{and}\quad g=w^2+t^2y+z^3y+\cdots\]
 The blow up is locally given by $u^\frac{1}{4}x,u^\frac{1}{4}s,u^\frac{3}{4}w$, which immediately reads $m_f=1/2$. On the other hand $m_g=1/2$ if and only if we have (at least) one of the monomials $z^3x^2$ or $z^2xs$ is in $g$. Note that the first one can be killed using the term $z^3y$. 
 
 Suppose, first, that $z^2xs\notin g$, which implies $m_g=3/2$ as $w^2\in g$. In this case, the Cox ring of the toric ambient space $T$ is 
 \[\left(\begin{array}{ccccccc}
u&\z&\s&\t&\w&\x&\y\\
0&4&5&6&7&1&2\\
1&1&1&1&1&0&-1\end{array}\right)\]
and $Y$ is given by
 \[\overline{f}=\s^2+\t\z+\cdots\quad\text{and}\quad \overline{g}=\w^2+u\t^2\y+\z^3\y+\cdots\]
 It can be checked that $Y$ 2-ray follows $T$ and, in particular, $\MMob(Y)=\left<(4,1),(1,0)\right>$. By adjunction formula $-K_Y\in|\mathcal{O}_Y(1,0)|$, which shows that $-K_Y$ is in the boundary of the mobile cone.
 
 The case $z^2xs\in g$ will be treated in Subsection\,\ref{elliptic} as Family No\,$51^*$.

 \subsubsection{Elliptic involutions}\label{elliptic}
 
 \paragraph*{\bf Family No.\,$31^*$}$\quad\frac{1}{3}(1,1,2)\in X_{8,10}\subset\PP(1,2,3,4,4,5)$\\
 
 Recall that $m_f=2/3$ and $m_g=4/3$.
 \[f=z^2y+ts+tzx+\cdots\quad\text{and}\quad g=w^2+z^2s+t^2y+s^2y+\cdots\]
 Moreover, at least one of the monomials $x^8$, $x^5z$ or $x^3w$ appear in $f$ with non-zero coefficient  (because of the quasi-smoothness condition) and also $f$ does not contain the following monomials: $zw$ (Okada's speciality condition), $s^2$, $t^2$ (without loss of generality assumption on the position of $1/4(1,1,3)$ points), $y^4$ (assumption that $p_y\in X$), $z^2x^2$, $x^6y$ (these two can be killed using the term $z^2y$) and $x^4t$ (which can be killed using the term $ts$). And similarly, $g$ does not contain $z^2t$ (Okada's speciality condition), $z^3x$ (this can be killed using the term $z^2s$) and $y^5$ (following the assumption that $p_y\in X$).
 
 The Cox ring of $T$ is
 \[\left(\begin{array}{ccccccc}
u&\z&\t&\w&\x&\y&\s\\
0&3&4&5&1&2&4\\
1&1&1&1&0&0&0\end{array}\right)\]

Note that $\overline{f}$, the proper transform of $f$ under the blow up, is contained in the ideal $(u,z)$. As in \ref{special(1112)}, we need to unproject $Y$ from the (fake) divisor $D=\{u=z=0\}$. This is done by introducing the new variable
\[r=\frac{zy+xt}{u}=\frac{ts+\cdots}{z}\]
The Cox ring of the new toric variety $T'$ is now
 \[\left(\begin{array}{cccccccc}
u&\z&\t&\w&\x&\y&\s&r\\
0&3&4&5&1&2&4&5\\
1&1&1&1&0&0&0&0\end{array}\right)\]
and $Y$ is defined as a complete intersection in $T'$ by three equations

\[\begin{array}{cl}
\circled{1}& ur=zy+xt\\
\circled{2}&rz=ts+\cdots\\
\circled{3}&w^2+z^2s+t^2y+u^2s^2y+\cdots\end{array}\]

The 2-ray game for $T'$ goes as
\[
\xymatrixcolsep{1.5pc}\xymatrixrowsep{2.3pc}
\xymatrix{
&T^\prime\ar^{\text{FLIP}}@{-->}[rr]\ar_\Phi[ld]&&T_1\ar^{\text{FLIP}}@{-->}[rr]&&T_2\ar^\Psi[rd]&\\
\PP&&&&&&\PP(1,2,4,5)}\]

The first FLIP restricts to a $(-1,-4,1,5)$ anti-flip on $Y$ to $Y_1$, and the second one reads an isomorphism on the 3-fold (it happens away from the 3-folds as $w^2\in\overline{g}$). The map $\Psi$ is a fibration over $\PP(1,2,4,5)$ with $\PP^3$ fibres. We show that this restricts to a generically 2-to-1 map on $Y_1$, and the contracted locus in $Y_1$ is 1-dimensional.

Note that if $r\neq 0$, on the base of the fibration, then it follows from $\circled{1}$ and $\circled{2}$ that on the fibre of each point the variables $u$ and $z$ can be eliminated, so that the fibre is the solution of a quadratic equation (what remains from $\circled{3}$) in $\PP^1_{t:w}$.

Note also that quasi-smoothness also implies that at least one of the monomials $x^8$, $x^5z$ or $x^3w$ must appear in $f$, which provides at least one of the monomials $ux^8$, $x^5z$ or $x^3w$ must appear in $\circled{2}$.  Using this, and the equations $\circled{1}$ and $\circled{2}$, one can check that in the locus $\{r=0\}$, if $x\neq 0$, the fibre of each point contains exactly two points (with multiplicity of course). So it is reduced to study the fibres above the line $\PP(2,4)$, that is $\{x=r=0\}$. Similarly, the fibre of any point except $p_y$ and $p_s$ contains only two points, while over each of these two points we have a conic curve. By Stein factorisation, this map factors into a small contraction (of these two conics) followed by a 2-to-1 map to $\PP(1,2,4,5)$. Note that the small contraction is $K$-trivial, as $-K_Y\in\mathcal{O}_Y(1,0)$, hence these curves must be flopped. But the flop corresponds simply to switching the fibres of the 2-to-1 map, and the 2-ray game continues by symmetrically doing this process backwards to reach $X$ again, hence an elliptic involution. \\

\paragraph*{\bf Family No.\,$51^*$}$\quad\frac{1}{4}(1,1,3)\in X_{10,14}\subset\PP(1,2,4,5,6,7)$\\
 Recall that 
  \[f=s^2+tz+\cdots\quad\text{and}\quad g=w^2+t^2y+z^3y+z^2xs+\cdots\]
  and $m_f=m_g=1/2$, and the Cox ring of the toric space $T$ is
   \[\left(\begin{array}{ccccccc}
u&\z&\s&\t&\w&\x&\y\\
0&4&5&6&7&1&2\\
1&1&1&1&1&0&0\end{array}\right)\]
 and $Y$ is defined by
 \[\overline{f}=\s^2+\t\z+\cdots\quad\text{and}\quad \overline{g}=\z^2(\x\s+\z\y)+u(\w^2+\t^2y+\cdots)\]
which indicates that we must unproject (using $\overline{g}$) to introduce a new variable 
\[r=\frac{\x\s+\z\y}{u}=\frac{\w^2+\t^2\y+\cdots}{-z^2}\]
which embeds $Y$ defined by equations
\[\left\{\begin{array}{l}
ru=\x\s+z\y\\
r\z^2+\w^2+\t^2\y+\cdots=0\\
\s^2+\z\t+\cdots
\end{array}
\right.\]
into the toric space $T'$ with Cox ring
\[\left(\begin{array}{cccccccc}
u&\z&\s&\t&\w&\x&\y&r\\
0&4&5&6&7&1&2&6\\
1&1&1&1&1&0&0&0\end{array}\right)\]
The first step of the toric 2-ray game restricted to $Y$ is an isomorphism, as $s^2$ appears in the equations above. However, the second step does not restrict to a FLIP on $Y$, this is because the first equation and the third belong to the ideal $(u,z,s)$, which is indeed a component of the irrelevant ideal of the toric space $T'_1$, after the FLIP. In order to unproject this (fake) divisor we introduce a new variable $\eta$ as follows (see\,\cite[\S4]{graded} for details on this type of unprojection): first rewrite the third equation above as $uF+\z G+\s H=0$. Then write
\[\eta=\frac{\y H-\x G}{u}=\frac{rH-\x F}{\z}=\frac{rG-\y F}{\s}\]
The new toric space $T''$ is given by the Cox ring
\[\left(\begin{array}{ccccccccc}
u&\z&\s&\t&\w&\x&\y&r&\eta\\
0&4&5&6&7&1&2&6&7\\
1&1&1&1&1&0&0&0&0\end{array}\right)\]
where $Y$ is now defined by the equations
\[\left\{\begin{array}{l}
ru=\x\s+z\y\\
r\z^2+\w^2+\t^2\y+\cdots=0\\
\s^2+\z\t+\cdots\\
u\eta=\y H-\x G\\
\z\eta=rH-\x F\\
\s\eta=rG-\y F
\end{array}
\right.\]
Now the game follows and ends with a 2-to-1 map to $\PP1,2,6,7)$. As in the previous cases of an elliptic involution, one can check that it is 2-to-1 for all points with the difference that there is no curve being contracted, hence there is no flop involved. Switching the fibres of this map and going back in the direction of the 2-ray game completes the game for $Y$, and goes back to $Y$.

 \subsubsection{Failure of mobility}\label{f-mobility}
 
\begin{lm}\label{family-mob} Suppose $\{C_\alpha\}$ is a $1$-dimensional family of curves in $Y$, where $(Y,E)\rightarrow(X,\Gamma)$ is a ($\Q$-factorial) maximal extraction, that is an extremal extraction at a maximal centre, of a Fano variety $X$ at the centre $\Gamma$. If $C_\alpha .E>0$ and $C_\alpha .(-K)\leq 0$, then $Y$ cannot initiate a Sarkisov link for $X$.\end{lm}
\begin{proof} Suppose there is such Sarkisov link. It follows from Theorem\,\ref{idea-Fano} that $Y$ is a Mori dream space and $-K\in\Int\Mob(Y)$. This, together with the assumption that $C_\alpha.E>0$ imply that there are only finitely many curves that intersect $-K_Y$ non-positively. In fact, those that intersect negatively correspond to anti-flipping curves and those with trivial intersection are the flopping curves. This contradicts the assumption.
\end{proof}

\paragraph*{\bf Family No.\,$37$}$\quad\frac{1}{2}(1,1,1)\in X_{8,12}\subset\PP(1,2,3,4,5,6)$\\
Okada's assumption is that $s^2\notin f$, which forces $s^3\in g$ by quasi-smoothness assumption. We can assume, without loss of generality, that the point to be excluded is $p_y$, in other words $y^4\notin f$ and $y^6\notin g$. We also have
\[f=tz+\cdots\quad\text{and}\quad g=s^3+w^2+t^2y+z^2w+\cdots\]
Note that we have also assumed, without loss of generality, that one of the $1/3(1,1,2)$ points is the $p_z$ point. Clearly $m_f=m_g=1$. So we can write down the Cox ring of $T$ as
\[\left(\begin{array}{ccccccc}
u&\y&\t&\z&\w&\s&\x\\
0&2&5&3&6&4&1\\
1&1&2&1&2&1&0\end{array}\right)\]
and $Y$ is defined by
\[\overline{f}=\t\z+\cdots\quad\text{and}\quad \overline{g}=u^2\s^3+u\w^2+\t^2\y+u\z^2\w+\cdots\]
In usual circumstances, we would have done an unprojection using $\overline{g}$. But if we do so, then we realise that there are more and (possibly) more unprojections to be done. Instead, we use Lemma\,\ref{family-mob} in this case. Since $-K_Y\in|\mathcal{O}_Y(1,0)|$, the Cox ring of $T$ suggests that the family of curves must be contained in the locus $\{\x=0\}$. Indeed, we take 
\[C_{\lambda,\mu}=\{\x=\lambda\z^2+\mu\w=\overline{f}=\overline{g}=0\}\subset Y\]
Numerically, all these curves are proportional, so we only do the computation for $C=C_{0,1}$. 

Let us recall that the irrelevant ideal of $T$ is $I_T=(u,\y)\cap(\t,\z,\w,\s,\x)$. Setting $u=0$, that is $E=\{u=0\}$, forces $\y\neq 0$ indicated by the first component of the ideal $I_T$. Hence, we can set the second component of the $(\C^*)^2$ action, to set $\y=1$. Therefore, $C.E$ corresponds to the solution of
\[\overline{f}(0,1,\t,\z,0,0,\s,0)=\overline{g}(0,1,\t,\z,0,0,\s,0)=0\subset\PP_{\t:\z:\s}(1,1,2)\]
which gives $C.E=2$. 

On the other hand, $-2K_Y\sim D-E$, where $D$ is the divisor $\{\y=0\}$. Similar to $C\cdot E$, one can check that $C\cdot D=\frac{7}{5}$. In particular, we have that $C\cdot{-K_Y}<0$. Now it from Lemma\,\ref{family-mob} that $-K_Y$ is not mobile.

\bibliographystyle{amsplain}
\bibliography{bib}

\vspace{0.6cm}

School of Mathematics, University of Bristol, Bristol BS8 1TW, UK

e-mail: \url{h.ahmadinezhad@bristol.ac.uk}

\vspace{0.4cm}

Dipartimento di Matematica e Informatica, Universit\`{a} degli Stud\^{i} di Udine,

Via delle Scienze, 206, 33100 Udine, Italia

e-mail: \url{francesco.zucconi@uniud.it}

\end{document}